\documentclass{amsart}
\usepackage{amsmath}
\usepackage{amssymb}
\usepackage{latexsym}
\usepackage{color}
\usepackage[margin=1in]{geometry}
\usepackage{tikz}
\usetikzlibrary{math}
\usepackage[bookmarks=true, bookmarksopen=true, bookmarksdepth=3,bookmarksopenlevel=2, colorlinks=true, linkcolor=blue, citecolor=blue, filecolor=blue, menucolor=blue, urlcolor=blue]{hyperref}

\newtheorem{theorem}{Theorem}
\newtheorem{corollary}[theorem]{Corollary}
\newtheorem{definition}[theorem]{Definition}
\newtheorem{lemma}[theorem]{Lemma}
\newtheorem{proposition}[theorem]{Proposition}
\newtheorem{remark}[theorem]{Remark}

\numberwithin{theorem}{section}

\newcommand{\bfa}{\boldsymbol{a}}

\newcommand{\bfg}{\boldsymbol{g}}

\newcommand{\cA}{\mathcal{A}}
\newcommand{\cC}{\mathcal{C}}

\newcommand{\cG}{\mathcal{G}}
\newcommand{\cI}{\mathcal{I}}
\newcommand{\cP}{\mathcal{P}}

\newcommand{\cR}{\mathcal{R}}
\newcommand{\cS}{\mathcal{S}}

\newcommand{\QQ}{\mathbb{Q}}
\newcommand{\RR}{\mathbb{R}}

\newcommand{\ZZ}{\mathbb{Z}}

\newcommand{\into}{\hookrightarrow}

\newcommand{\kk}{\Bbbk}

\definecolor{darkgreen}{rgb}{0,0.7,0}

\title{Dominance Regions for Rank Two Cluster Algebras}

\author{Dylan Rupel}
\author{Salvatore Stella}
\address[D. Rupel]{\href{mailto:rupel.dylan@pusd.us}{\textup{\texttt{rupel.dylan@pusd.us}}}, Math Academy, Pasadena Unified School District, Pasadena, CA 91109 USA}
\address[S. Stella]{\href{mailto:salvatore.stella@univaq.it}{\textup{\texttt{salvatore.stella@univaq.it}}}, Dipartimento di Ingegneria e Scienze dell'Informazione e Matematica, Università degli Studi dell'Aquila, IT}
\thanks{Salvatore Stella was partially supported by the University of Leicester and the University of Rome ``La Sapienza''.}

\begin{document}

\begin{abstract}
  We study the polygons defining the dominance order on $\bfg$-vectors in cluster algebras of rank 2.
\end{abstract}
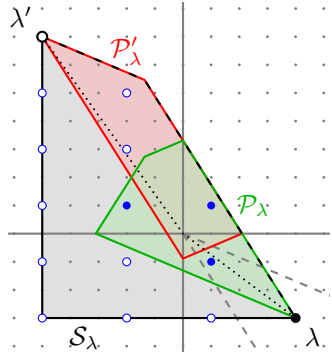
\maketitle
\vspace{-1em}
\begin{figure}[h!]
  \centering
    \newcommand{\setconstants}{
      \tikzmath{\b = 3; \c = 6 / \b; \s = sqrt(\b * \c * (\b * \c -4)); \bcmo = \b * \c - 1; \bcms = \b * \c - \s; \bcps = \b * \c + \s;}
      \tikzmath{ \xmin = -6; \xmax = 5; \ymin = -4; \ymax = 8;}
      \clip ({\xmin-0.2},{\ymin-0.2}) rectangle ({\xmax+0.2},{\ymax+0.2});
    }
    \newcommand{\boundingrays}{
      \draw [color=gray, thick] (2*\xmin,0) -- (2*\xmax,0);
      \draw [color=gray, thick] (0,2*\ymin) -- (0, 2*\ymax);
      \draw [color=gray, thick, dashed] (0,0) -- (2*\xmax, -\xmax * \bcms / \b); 
      \draw [color=gray, thick, dashed] (0,0) -- ( - 4 * \b * \ymin / \bcps, 2*\ymin); 
    }
  \begin{tikzpicture}[scale=.3743]
      \setconstants;
      \tikzmath{ \x=4; \y=-3;};
      \fill [fill=black!30, fill opacity=0.4] (\x,\y) -- ( \x+\b*\y, \y) -- ( \x+\b*\y , -\c*\x - \b*\c*\y + \y) -- (\bcps*\x/\s + \bcps*\b*\y/2/\s, -\bcps*\c*\x/2/\s - \bcps*\b*\c*\y/2/\s + \bcps*\y/\s) -- cycle;
      \fill [fill=red!30, fill opacity=0.5]  ( \x+\b*\y , -\c*\x - \b*\c*\y + \y)  -- (\bcps*\x/\s + \bcps*\b*\y/2/\s, -\bcps*\c*\x/2/\s - \bcps*\b*\c*\y/2/\s + \bcps*\y/\s) -- ({\bcms/(2*\c)*\y+\x},0) --  (0,{\bcps/(2*\b)*(\x+\b*\y)+(-\c*\x - \b*\c*\y + \y)}) -- cycle;
      \fill [fill=darkgreen!30, fill opacity=0.5] (\x,\y) -- ({\x+\bcps/(2*\c)*\y},0) -- ({\bcps/(2*\s)*(2*\x+\b*\y)},{\bcps/(2*\s)*(\c*\x+2*\y)}) -- (0,{\bcps/(2*\b)*\x+\y}) -- cycle;
      \foreach \x in {\xmin,...,\xmax} {
        \foreach \y in {\ymin,...,\ymax} {
          \fill[color=gray] (\x,\y) circle (0.05);
        }
      };
      \boundingrays;
      \draw [red, thick, join=round]  ( \x+\b*\y , -\c*\x - \b*\c*\y + \y)  -- (\bcps*\x/\s + \bcps*\b*\y/2/\s, -\bcps*\c*\x/2/\s - \bcps*\b*\c*\y/2/\s + \bcps*\y/\s) -- ({\bcms/(2*\c)*\y+\x},0) -- (0,{\bcps/(2*\b)*(\x+\b*\y)+(-\c*\x - \b*\c*\y + \y)}) -- cycle;
      \draw [darkgreen, thick, join=round] (\x,\y) -- ({\x+\bcps/(2*\c)*\y},0) -- ({\bcps/(2*\s)*(2*\x+\b*\y)},{\bcps/(2*\s)*(\c*\x+2*\y)}) -- (0,{\bcps/(2*\b)*\x+\y}) -- cycle;
      \draw [thick, join=round] (\x , \y) -- ( \x+\b*\y, \y) -- ( \x+\b*\y , -\c*\x - \b*\c*\y + \y);;
      \tikzmath{ \gap=3; }
      \draw [thick, join=round, line cap=round, dash pattern=on 0pt off {\gap*\pgflinewidth}] ( \x+\b*\y , -\c*\x - \b*\c*\y + \y) -- (0,0) --  (\x , \y);
      \draw [thick, join=round, line cap=round, dash pattern=on {\gap*\pgflinewidth} off {2*\gap*\pgflinewidth}] ( \x+\b*\y , -\c*\x - \b*\c*\y + \y) --  (\bcps*\x/\s + \bcps*\b*\y/2/\s, -\bcps*\c*\x/2/\s - \bcps*\b*\c*\y/2/\s + \bcps*\y/\s) --  (\x , \y);
      \draw [thick, fill=white]  ( \x+\b*\y , -\c*\x - \b*\c*\y + \y) circle (5pt) node[above left]{$\lambda'$};
      \fill (\x , \y)  circle (5pt) node[below right]{$\lambda$};
      \foreach \x/\y in { 1/-3, -2/-3, -5/-3, -2/-1, -5/-1, -5/1, -2/3, -5/3, -2/5, -5/5 } {
        \draw[color=blue, fill=white] (\x,\y) circle (4pt);
      };
      \foreach \x/\y in { 1/-1, 1/1, -2/1 } {
        \fill[color=blue] (\x,\y) circle (4pt);
      };
      \draw (-3.5,-3.7) node{$\cS_\lambda$};
      \draw [darkgreen](2.5,1) node{$\cP_\lambda$};
      \draw [red](-2,6.5) node{$\cP'_\lambda$};
  \end{tikzpicture}
\end{figure}
\vspace{-1em}
  \section{Introduction}
  Cluster algebras were introduced by Fomin and Zelevinsky as a tool in the study of Lusztig's dual canonical bases.
  Since their inception they have found application in a variety of different areas in mathematics, nevertheless a fundamental problem in the theory remains constructing bases with ``good'' properties. 
  
  Over time several bases for cluster algebras have been described in different generalities \cite{BZ14,CI12,DT13,Dup11,Dup12,GHKK18,LLZ14,MSW13,Pla13,SZ04,T14}.
  All of their elements share the property of being ``pointed''; this turned out to be a desirable feature for a basis to have and a natural question is to find all bases enjoying this property.
  Recently Qin studied the deformability of pointed bases whenever the cluster algebra has full rank \cite{Qin19}. 
  As no explicit calculation is carried out in his work, we aim here at describing explicitly the space of deformability for pointed bases in rank two, i.e. when clusters contain a pair of mutable cluster variables.
  In this setting, frozen variables do not carry any additional information so we will work in the coefficient-free case. 

  Fix integers $b,c>0$.
  The cluster algebra $\cA(b,c)$ is the $\ZZ$-subalgebra of $\QQ(x_0,x_1)$ generated by the \emph{cluster variables} $x_m$, $m\in\ZZ$, defined recursively by
  \[
    x_{m-1}x_{m+1}=\begin{cases} x_m^b+1 & \text{if $m$ is even;}\\ x_m^c+1 & \text{if $m$ is odd.} \end{cases}
  \]
  By the Laurent Phenomenon \cite{FZ02}, each $x_m$ is actually an element of $\ZZ[x_k^{\pm1},x_{k+1}^{\pm1}]$ for any $k\in\ZZ$.
  Moreover, these Laurent polynomials are known to have positive coefficients \cite{GHKK18,LLZ14,LS15}.

  An element of $\cA(b,c)$ has \emph{$\bfg$-vector} $\lambda=(\lambda_0,\lambda_1)\in\ZZ^2$, with respect to the embedding ${\cA(b,c)\into\ZZ[x_0^{\pm1},x_{1}^{\pm1}]}$, if it can be written in the form
  \begin{equation}
    \label{eq:pointed}
    x_0^{\lambda_0}x_1^{\lambda_1}\sum\limits_{\alpha_0,\alpha_1 \ge 0} \rho_{\alpha_0,\alpha_1} x_0^{-b\alpha_0} x_1^{c\alpha_1}
  \end{equation}
  with $\rho_{0,0}=1$.
  An element of $\cA(b,c)$ is \emph{pointed} if an analogous structure reproduces after expanding in terms of any pair $\{x_k,x_{k+1}\}$; a basis is pointed if it consists entirely of pointed elements.

  Important examples of pointed elements are the \emph{cluster monomials} of $\cA(b,c)$, i.e.~the elements of the form $x_k^{\alpha_k}x_{k+1}^{\alpha_{k+1}}$ for some $k\in\ZZ$ and $\alpha_k,\alpha_{k+1}$ non-negative integers.
  Indeed, by \cite[Lemma 3.4.12]{Qin19} cluster monomials are part of any pointed basis in any (upper) cluster algebra of full rank.
  We achieve the same conclusion in our setting by elementary calculations (cf. Lemma~\ref{le:cluster monomials}).

  Elements of any pointed basis are parametrized by $\ZZ^2$ thought of as the collection of possible $\bfg$-vectors.
  Qin introduced a partial order $\preceq$ on $\bfg$-vectors called the \emph{dominance order}, refining the order used in \cite[Proposition 4.3]{CILS15}, and showed that it provides a characterization of pointed bases.
  We restate his results in the generality needed for this paper and using our notation.
  \begin{theorem}
    \label{th:dominance}
    \cite[Theorem 1.2.1]{Qin19}
    Let $\{x_\lambda\}$ and $\{y_\lambda\}$ be pointed bases of $\cA(b,c)$.
    Then for each $\lambda\in\ZZ^2$, there exist scalars $q_{\lambda,\mu}$ for $\mu\prec\lambda$ such that
    \[y_\lambda=x_\lambda+\sum_{\mu\prec\lambda} q_{\lambda,\mu} x_\mu.\]
    Moreover, having fixed a reference pointed basis $\{x_\lambda\}$, any choice of scalars $q_{\lambda,\mu}$ as above provides a pointed basis of $\cA(b,c)$.
  \end{theorem}

  When $bc\le3$, the cluster algebra $\cA(b,c)$ will be of finite-type and cluster monomials form its only pointed basis.
  We therefore assume that $bc\ge4$.
  Write $\cI \subset \RR^2$ for the \emph{imaginary cone} (positively) spanned by the vectors $\big(2b,-bc\pm\sqrt{bc(bc-4)}\big)$.
  Lattice points outside of $\cI$ are precisely the $\bfg$-vectors of cluster monomials in $\cA(b,c)$.

  We give an explicit description of the dominance relation among $\bfg$-vectors.
  Specifically, we show that the $\bfg$-vector $\lambda$ dominates the collection of $\bfg$-vectors of the form $\lambda+(b \alpha_0 ,c \alpha_1)$, $\alpha_0,\alpha_1\in\ZZ$, inside its \emph{dominance region} $\cP_\lambda$ (cf. Definition~\ref{def:dominance}).
  \begin{theorem}
    \label{th:dominance inequalities}
    If $\lambda$ lies outside of $\cI$, then the dominance region $\cP_\lambda$ is the point $\lambda$.
    Otherwise the dominance region $\cP_\lambda$ of the $\bfg$-vector $\lambda=(\lambda_0,\lambda_1)$ is the polygon consisting of those $\mu=(\mu_0,\mu_1)\in\RR^2$ satisfying the following inequalities:
    {
      \everymath={\displaystyle}
      \def\arraystretch{2.2}
      \[
        \begin{array}{rcccl}
          0 & \!\leq\! & \frac{b c-\sqrt{b c (b c-4)}}{2 b}(\mu_0-\lambda_0)+(\mu_1-\lambda_1) & \!\leq\! & -c\lambda_0-\frac{b c+\sqrt{b c (b c-4)}}{2b}b\lambda_1;
          \\
          0 & \!\leq\! & -\frac{b c-\sqrt{b c (b c-4)}}{2 b}(\mu_0-\lambda_0)+(\mu_1-\lambda_1) & \!\leq\! & c\lambda_0;
          \\
          0 & \!\leq\! &  -(\mu_0-\lambda_0)-\frac{b c-\sqrt{b c (b c-4)}}{2 c}(\mu_1-\lambda_1) & \!\leq\! & \frac{b c+\sqrt{b c (b c-4)}}{2c}c\lambda_0+b\lambda_1;
          \\
          0 & \!\leq\! & -(\mu_0-\lambda_0) + \frac{b c-\sqrt{b c (b c-4)}}{2 c} (\mu_1-\lambda_1) & \!\leq\! & -b \lambda_1.
        \end{array}
      \]
    }
  \end{theorem}
  \begin{remark}
    Explicit calculations reveal that for certain pairs of notable bases (e.g. greedy and triangular) most of the coefficients $q_{\lambda,\mu}$ are zero.
    Understanding this phenomenon might be worth further study.
  \end{remark}

  \begin{corollary}
    \label{cor:dominance vertices}
    There are six classes of dominance polygons.
    \begin{enumerate}
      \item If $\lambda$ lies outside of $\cI$, then $\cP_\lambda$ is the point $\lambda$. 
      \item If $\lambda$ lies in the cone spanned by the vectors $(2b,-bc-\sqrt{bc(bc-4)})$ and $(2,-c)$, then $\cP_\lambda$ is the trapezoid with  vertices $\lambda$, $\big(0,\frac{bc+\sqrt{bc(bc-4)}}{2b}\lambda_0+\lambda_1\big)$, $-\frac{bc+\sqrt{bc(bc-4)}}{2c}\big(\frac{bc+\sqrt{bc(bc-4)}}{2b}\lambda_0+\lambda_1,0\big)$, and ${\frac{bc+\sqrt{bc(bc-4)}}{2\sqrt{bc(bc-4)}}\big(-(bc-2)\lambda_0-b\lambda_1,c\lambda_0+2\lambda_1\big)}$.
      \item If $\lambda$ lies on the ray spanned by $(2,-c)$, then $\cP_\lambda$ is the triangle with vertices $\lambda$, $\big(0,\frac{bc+\sqrt{bc(bc-4)}}{2b}\lambda_0+\lambda_1\big)$, and $\big(\lambda_0+\frac{bc+\sqrt{bc(bc-4)}}{2c}\lambda_1,0\big)$.
      \item If $\lambda$ lies in the cone spanned by the vectors $(2,-c)$ and $(b,-2)$, then $\cP_\lambda$ is the kite with  vertices $\lambda$, $\big(0,\frac{bc+\sqrt{bc(bc-4)}}{2b}\lambda_0+\lambda_1\big)$, $\frac{bc+\sqrt{bc(bc-4)}}{2\sqrt{bc(bc-4)}}(2\lambda_0+b\lambda_1,c\lambda_0+2\lambda_1)$, and $\big(\lambda_0+\frac{bc+\sqrt{bc(bc-4)}}{2c}\lambda_1,0\big)$.
      \item If $\lambda$ lies on the ray spanned by $(b,-2)$, then $\cP_\lambda$ is the triangle with vertices $\lambda$, $\big(0,\frac{bc+\sqrt{bc(bc-4)}}{2b}\lambda_0+\lambda_1\big)$, and $\big(\lambda_0+\frac{bc+\sqrt{bc(bc-4)}}{2c}\lambda_1,0\big)$.
      \item If $\lambda$ lies in the cone spanned by the vectors $(b,-2)$ and $(2b,-bc+\sqrt{bc(bc-4)})$, then $\cP_\lambda$ is the trapezoid with vertices $\lambda$, $\frac{bc+\sqrt{bc(bc-4)}}{2\sqrt{bc(bc-4)}}\big(2\lambda_0+b\lambda_1,-c\lambda_0-(bc-2)\lambda_1\big)$, $-\frac{bc+\sqrt{bc(bc-4)}}{2b}\big(0,\lambda_0+\frac{bc+\sqrt{bc(bc-4)}}{2c}\lambda_1\big)$, and $\big(\lambda_0+\frac{bc+\sqrt{bc(bc-4)}}{2c}\lambda_1,0\big)$.
    \end{enumerate}
  \end{corollary}

  \begin{remark}
    Note that the rays which separate the regions inside $\cI$ correspond exactly to the columns of the associated Cartan matrix. 
    Moreover the expression for the dominance regions along those rays coincide.
    These unexpected coincidences are one of our reasons for deciding to write down these results.
  \end{remark}

  Given a Laurent polynomial in $\ZZ[x_0^{\pm1},x_1^{\pm1}]$, its \emph{support} is the set of its exponent vectors inside $\ZZ^2$.
  Given a $\bfg$-vector $\lambda$, in the next result we identify a polygon $\cS_\lambda$ whose lattice points of the form $\lambda+(b \alpha_0 ,c \alpha_1)$ for~$\alpha_0,\alpha_1\in\ZZ$ give the maximum possible support of a pointed basis element $x_\lambda$. 
  \begin{theorem}
    \label{th:maximum support}
    Let $\lambda=(\lambda_0,\lambda_1)\in\ZZ^2$ be the $\bfg$-vector for a pointed basis element $x_\lambda$.
    \begin{enumerate}
      \item If $\lambda$ lies outside of $\cI$, then the support of $x_\lambda$ is precisely the points of the form $\lambda+(b \alpha_0 ,c \alpha_1)$, $\alpha_0,\alpha_1\in\ZZ$, inside the region $\cS_\lambda$ given as follows:
        \begin{enumerate}
          \item If $0\le\lambda_0,\lambda_1$, then $\cS_\lambda$ is just the point $\lambda$.
          \item If $\lambda_1 < 0$ and $0\le\lambda_0+b\lambda_1$, then $\cS_\lambda$ is the segment joining $\lambda$ and $(\lambda_0+b\lambda_1,\lambda_1)$.
          \item If $\lambda_0 < 0$ and $0\le\lambda_1$, then $\cS_\lambda$ is the segment joining $\lambda$ and $(\lambda_0,-c\lambda_0+\lambda_1)$.
          \item If $\lambda_0,\lambda_1 < 0$, then $\cS_\lambda$ has  vertices $\lambda$, $(\lambda_0+b\lambda_1,\lambda_1)$, $\big(\lambda_0+b\lambda_1,-c\lambda_0-(bc-1)\lambda_1\big)$, and ${(\lambda_0,-c\lambda_0+\lambda_1)}$.
          \item If $\lambda_0 > 0$, $\lambda_0+b\lambda_1 < 0$, and $-c\lambda_0-(bc-1)\lambda_1\le 0$, then $\cS_\lambda$ has  vertices $\lambda$, $(\lambda_0+b\lambda_1,\lambda_1)$, $\big(\lambda_0+b\lambda_1,-c\lambda_0-(bc-1)\lambda_1\big)$, and $\big((bc+1)\lambda_0-b^2c\lambda_1,-c\lambda_0-(bc-1)\lambda_1\big)$.
          \item If $\lambda_0 > 0$, $\lambda_0+b\lambda_1 < 0$, and $0 < -c\lambda_0-(bc-1)\lambda_1$, then $\cS_\lambda$ has  vertices $\lambda$, $(\lambda_0+b\lambda_1,\lambda_1)$, $\big(\lambda_0+b\lambda_1,-c\lambda_0-(bc-1)\lambda_1\big)$, and $(0,0)$.  Here the point $(0,0)$ and its adjacent open segments are excluded from $\cS_\lambda$.
        \end{enumerate}
      \item If $\lambda$ lies inside of $\cI$, then the support of $x_\lambda$ is contained in $\cS_\lambda$ with  vertices $\lambda$, $(\lambda_0+b\lambda_1,\lambda_1)$, $\big(\lambda_0+b\lambda_1,-c\lambda_0-(bc-1)\lambda_1\big)$, and $\frac{bc+\sqrt{bc(bc-4)}}{2\sqrt{bc(bc-4)}}\big(2\lambda_0+b\lambda_1,-c\lambda_0-(bc-2)\lambda_1\big)$.
        Moreover, there is an element pointed at $\lambda$ whose support is precisely the points of the form $\lambda+(b \alpha_0 ,c \alpha_1)$, $\alpha_0,\alpha_1\in\ZZ$, inside $\cS_\lambda$.
    \end{enumerate}
  \end{theorem}

\begin{figure}[h!]
  \centering
    \newcommand{\setconstants}{
      \tikzmath{\b = 3.5; \c = 4.4 / \b; \s = sqrt(\b * \c * (\b * \c -4)); \bcmo = \b * \c - 1; \bcms = \b * \c - \s; \bcps = \b * \c + \s;}
      \tikzmath{ \xmin = -6; \xmax = 6; \ymin = -6; \ymax = 6;}
      \clip (\xmin,\ymin) rectangle (\xmax,\ymax);
    }
    \newcommand{\boundingrays}{
      \draw [color=gray, thick] (2*\xmin,0) -- (2*\xmax,0);
      \draw [color=gray, thick] (0,2*\ymin) -- (0, 2*\ymax);
      \draw [color=gray, thick] (0,0) -- (2*\xmax, -2*\xmax / \b);
      \draw [color=gray, thick] (0,0) -- (2*\xmax, -2*\xmax * \c / \bcmo );
      \draw [color=gray, thick, dashed] (0,0) -- (2*\xmax, -\xmax * \bcms / \b); 
      \draw [color=gray, thick, dashed] (0,0) -- ( - 4 * \b * \ymin / \bcps, 2*\ymin); 
    }
  \begin{tikzpicture}[scale=.3]
    \begin{scope}[shift={(0,0)}]
      \setconstants;
      \draw[draw=none, fill=gray!20] (0,0) -- (2*\xmax,0) -- (2*\xmax,2*\ymax) --  (0,2*\ymax) -- cycle; 
      \boundingrays;
      \tikzmath{ \t=3; \r=1; \q=1; \xx = \r; \yy=\q; \x={\t*\xx/sqrt((\xx)^2+(\yy)^2)}; \y={\t*\yy/sqrt((\xx)^2+(\yy)^2)}; }
      \fill (\x,\y) circle (5pt);
      \draw (\xmin+1.5,\ymin+1) node{(1.a)};
    \end{scope}
    \begin{scope}[shift={(14,0)}]
      \setconstants;
      \draw[draw=none, fill=gray!20] (0,0) -- (2*\xmax,0) -- (4*\xmax,-2*\xmax / \b) --  (2*\xmax,-2*\xmax / \b) -- cycle; 
      \boundingrays;
      \tikzmath{ \t=5; \r=2; \q=1; \xx = \r + \b*\q; \yy=-\q; \x={\t*\xx/sqrt((\xx)^2+(\yy)^2)}; \y={\t*\yy/sqrt((\xx)^2+(\yy)^2)}; }
      \fill (\x, \y) circle (5pt);
      \draw [thick, join=round] (\x, \y) -- ({\x+\b * \y},\y);
      \draw [thick, fill=white] ({\x+\b * \y},\y) circle (5pt);
      \draw (\xmin+1.5,\ymin+1) node{(1.b)};
    \end{scope}
    \begin{scope}[shift={(28,0)}]
      \setconstants;
      \draw[draw=none, fill=gray!20] (0,0) -- (2*\xmin,0) -- (2*\xmin,2*\ymax) --  (0,2*\ymax) -- cycle; 
      \boundingrays;
      \tikzmath{ \t=3; \r=1; \q=1; \xx=-\r; \yy=\q; \x={\t*\xx/sqrt((\xx)^2+(\yy)^2)}; \y={\t*\yy/sqrt((\xx)^2+(\yy)^2)}; }
      \fill (\x, \y) circle (5pt);
      \draw [thick, join=round] (\x, \y) -- (\x,{\y - \c * \x});
      \draw [thick, fill=white] (\x,{\y - \c * \x}) circle (5pt);
      \draw (\xmin+1.5,\ymin+1) node{(1.c)};
    \end{scope}
    \begin{scope}[shift={(42, 0)}]
      \setconstants;
      \draw[draw=none, fill=gray!20] (0,0) -- (2*\xmin,0) -- (2*\xmin,2*\ymin) --  (0,2*\ymin) -- cycle; 
      \boundingrays;
      \tikzmath{ \t=1; \r=1; \q=1; \xx=-\r; \yy=-\q; \x={\t*\xx/sqrt((\xx)^2+(\yy)^2)}; \y={\t*\yy/sqrt((\xx)^2+(\yy)^2)}; }
      \fill (\x, \y) circle (5pt);
      \draw [thick, join=round] (\x, \y) -- ({\x + \b * \y}, \y) -- ({\x + \b * \y}, {\y -\c*(\x + \b * \y)}) -- (\x, {\y - \c*\x}) -- cycle;
      \draw [thick, fill=white] ({\x + \b * \y}, {\y -\c*(\x + \b * \y)})  circle (5pt);
      \draw (\xmin+1.5,\ymin+1) node{(1.d)};
    \end{scope}
    \begin{scope}[shift={(0,-14)}]
      \setconstants;
      \draw[draw=none, fill=gray!20] (0,0) -- (2*\xmax, -2*\xmax / \b) -- (4*\xmax, -2*\xmax / \b-2*\xmax * \c / \bcmo ) --  (2*\xmax, -2*\xmax * \c / \bcmo ) -- cycle; 
      \boundingrays;
      \tikzmath{ \t=5; \r=1; \q=1; \xx={\r*\b+\q*(\b*\c-1) }; \yy = {-\r -\c*\q};  \x={\t*\xx/sqrt((\xx)^2+(\yy)^2)}; \y={\t*\yy/sqrt((\xx)^2+(\yy)^2)}; }
      \fill (\x,\y) circle (5pt);
      \draw [thick, join=round] (\x,\y) -- ( \x+\b*\y , \y) -- ( \x+\b*\y , -\c*\x - \b*\c*\y + \y) -- ( \b*\c*\x + \x + \b*\b*\c*\y,  -\c*\x - \b*\c*\y + \y) --  cycle;
      \draw [thick, fill=white] ( \x+\b*\y , -\c*\x - \b*\c*\y + \y) circle (5pt);
      \draw (\xmin+1.5,\ymin+1) node{(1.e)};
    \end{scope}
    \begin{scope}[shift={(14,-14)}]
      \setconstants;
      \draw[draw=none, fill=gray!20] (0,0) --  (2*\xmax, -2*\xmax * \c / \bcmo ) -- (4*\xmax, -2*\xmax * \c / \bcmo -\xmax * \bcms / \b) -- (2*\xmax, -\xmax * \bcms / \b) -- cycle; 
      \boundingrays;
      \tikzmath{ \t=5; \r=1; \q=2; \xx={2*\r*\b+\q*(\b*\c-1) }; \yy = {-\r*\bcms -\c*\q};  \x={\t*\xx/sqrt((\xx)^2+(\yy)^2)}; \y={\t*\yy/sqrt((\xx)^2+(\yy)^2)}; }
      \fill (\x,\y) circle (5pt);
      \draw [thick, join=round] (\x , \y) -- ( \x+\b*\y , \y) -- ( \x+\b*\y , -\c*\x - \b*\c*\y + \y);
      \tikzmath{ \gap=3.1; }
      \draw [thick, join=round, line cap=round, dash pattern=on 0pt off {\gap*\pgflinewidth} on {\gap*\pgflinewidth} off {\gap*\pgflinewidth}] ( \x+\b*\y , -\c*\x - \b*\c*\y + \y) -- (0,0) --  (\x , \y);
      \draw [thick, fill=white] ( \x+\b*\y , -\c*\x - \b*\c*\y + \y) circle (5pt);
      \draw (\xmin+1.5,\ymin+1) node{(1.f)};
    \end{scope}
    \begin{scope}[shift={(28,-14)}]
      \setconstants;
      \draw[draw=none, fill=gray!20] (0,0) -- (0, 2*\ymin) -- ( - 4 * \b * \ymin / \bcps, 2*\ymin +2*\ymin) --  ( - 4 * \b * \ymin / \bcps, 2*\ymin) -- cycle; 
      \boundingrays;
      \tikzmath{ \t=2; \r=0.3; \q=1; \xx={2*\r*\b }; \yy = {-\r*\bcps -\q};  \x={\t*\xx/sqrt((\xx)^2+(\yy)^2)}; \y={\t*\yy/sqrt((\xx)^2+(\yy)^2)}; }
      \fill (\x,\y) circle (5pt);
      \draw [thick, join=round] (\x , \y) -- ( \x+\b*\y , \y) -- ( \x+\b*\y , -\c*\x - \b*\c*\y + \y);
      \tikzmath{ \gap=3.1; }
      \draw [thick, join=round, line cap=round, dash pattern=on 0pt off {\gap*\pgflinewidth} on {\gap*\pgflinewidth} off {\gap*\pgflinewidth}] ( \x+\b*\y , -\c*\x - \b*\c*\y + \y) -- (0,0) --  (\x , \y);
      \draw [thick, fill=white] ( \x+\b*\y , -\c*\x - \b*\c*\y + \y) circle (5pt);
      \draw (\xmin+1.5,\ymin+1) node{(1.f)};
    \end{scope}
    \begin{scope}[shift={(42,-14)}]
      \setconstants;
      \draw[draw=none, fill=gray!20] (0,0) -- (2*\xmax, -\xmax * \bcms / \b) -- ( - 4 * \b * \ymin / \bcps + 2*\xmax, 2*\ymin  -\xmax * \bcms / \b) -- ( - 4 * \b * \ymin / \bcps, 2*\ymin) -- cycle; 
      \boundingrays;
      \tikzmath{ \t=5; \r=1; \q=2; \xx={2*(\r+\q)*\b }; \yy = {-\r*\bcps -\q*\bcms}; \x={\t*\xx/sqrt((\xx)^2+(\yy)^2)}; \y={\t*\yy/sqrt((\xx)^2+(\yy)^2)};  }
      \fill (\x , \y)  circle (5pt);
      \draw [thick, join=round] (\x , \y) -- ( \x+\b*\y, \y) -- ( \x+\b*\y , -\c*\x - \b*\c*\y + \y);;
      \tikzmath{ \gap=3.1; }
      \draw [thick, join=round, line cap=round, dash pattern=on 0pt off {\gap*\pgflinewidth}] ( \x+\b*\y , -\c*\x - \b*\c*\y + \y) -- (0,0) --  (\x , \y);
      \draw [thick, join=round, line cap=round, dash pattern=on {\gap*\pgflinewidth} off {2*\gap*\pgflinewidth} ] ( \x+\b*\y , -\c*\x - \b*\c*\y + \y) --  (\bcps*\x/\s + \bcps*\b*\y/2/\s, -\bcps*\c*\x/2/\s - \bcps*\b*\c*\y/2/\s + \bcps*\y/\s) --  (\x , \y);
      \draw [thick, fill=white]  ( \x+\b*\y , -\c*\x - \b*\c*\y + \y) circle (5pt);
      \draw (\xmin+1.5,\ymin+1) node{(2)};
    \end{scope}
    \draw (21,-22) node{The maximal supports $\cS_\lambda$ as described in Theorem \ref{th:maximum support}.};
  \end{tikzpicture}
\end{figure}

The paper is organized as follows.
In Section~\ref{sec:chebyshev}, we collect useful results related to two-parameter Chebyshev polynomials which support our main calculations.
Section~\ref{sec:tropical} contains calculations related to the transformation of $\bfg$-vectors under mutations.
Section~\ref{sec:dominance inequalities} proves Theorem~\ref{th:dominance inequalities}.
Section~\ref{sec:dominance vertices} proves Corollary~\ref{cor:dominance vertices}.
Section~\ref{sec:maximum support} proves Theorem~\ref{th:maximum support}.
The paper ends with Section~\ref{sec:affine} interpreting the dominance polygons in terms of generalized minors in the cases where $b=c=2$.

\section{Chebyshev Polynomials}
  \label{sec:chebyshev}
  Define \emph{two-parameter Chebyshev polynomials} $u_i^\varepsilon$ for $i\in\ZZ$ and $\varepsilon\in\{\pm\}$ recursively by $u_0^\varepsilon=0$, $u_1^\varepsilon=1$, and
  \[u_{i+1}^\varepsilon=\begin{cases} bu_i^- -u_{i-1}^+ & \text{if $\varepsilon=+$;}\\ cu_i^+-u_{i-1}^- & \text{if $\varepsilon=-$.} \end{cases}\]
  \begin{remark}
    \label{rem:equivalences}
    Observe that, by easy inductions, we have $u_{-i}^\varepsilon=-u_i^\varepsilon$ for $i\in\ZZ$ and $u_{2j+1}^+=u_{2j+1}^-$ for $j\in\ZZ$.
  \end{remark}

  \begin{lemma}
    For $i,\ell\in\ZZ$ and $\varepsilon\in\{\pm\}$, we have
    \begin{equation}
      \label{eq:long Chebyshev recursion}
      u_{i+\ell}^\varepsilon=\begin{cases} u_{\ell+1}^\varepsilon u_i^{-\varepsilon}-u_\ell^{-\varepsilon} u_{i-1}^\varepsilon & \text{\upshape if $\ell$ is odd;}\\ u_{\ell+1}^{-\varepsilon} u_i^\varepsilon-u_\ell^\varepsilon u_{i-1}^{-\varepsilon} & \text{\upshape if $\ell$ is even.} \end{cases}
    \end{equation}
  \end{lemma}
  \begin{proof}
    We work by induction on $\ell$ for all $i$ simultaneously, the cases $\ell=0,1$ being tautological and reproducing the defining recursions, respectively.
    Using the claim for $\ell>0$ and then the defining recursion twice, $u_{i+\ell+1}^\varepsilon$ can be rewritten for $\ell$ odd as
    \[
      u_{\ell+1}^\varepsilon u_{i+1}^{-\varepsilon}-u_\ell^{-\varepsilon} u_i^\varepsilon
      =u_{\ell+1}^\varepsilon(u_2^{-\varepsilon}u_i^\varepsilon-u_{i-1}^{-\varepsilon}) -u_\ell^{-\varepsilon} u_i^\varepsilon 
      =(u_2^{-\varepsilon} u_{\ell+1}^\varepsilon-u_\ell^{-\varepsilon}) u_i^\varepsilon-u_{\ell+1}^\varepsilon u_{i-1}^{-\varepsilon}
      =u_{\ell+2}^{-\varepsilon} u_i^\varepsilon - u_{\ell+1}^\varepsilon u_{i-1}^{-\varepsilon}
    \]
    and for $\ell$ even as
    \[
      u_{\ell+1}^{-\varepsilon} u_{i+1}^\varepsilon-u_\ell^\varepsilon u_i^{-\varepsilon}
      =u_{\ell+1}^{-\varepsilon}(u_2^\varepsilon u_i^{-\varepsilon}-u_{i-1}^\varepsilon) -u_\ell^\varepsilon u_i^{-\varepsilon}
      =(u_2^\varepsilon u_{\ell+1}^{-\varepsilon}-u_\ell^\varepsilon) u_i^{-\varepsilon}-u_{\ell+1}^{-\varepsilon} u_{i-1}^\varepsilon
      =u_{\ell+2}^\varepsilon u_i^{-\varepsilon}-u_{\ell+1}^{-\varepsilon} u_{i-1}^\varepsilon.
    \]
    This gives the claimed recursion for $\ell+1$.
    These calculations can be reversed to show the result for $\ell<0$.
  \end{proof}

  \begin{lemma}
    \label{le:limits}
    We have
    \[\lim_{i\to\infty} \frac{u_i^-}{u_{i-1}^+}=\frac{bc+\sqrt{bc(bc-4)}}{2b}=\lim_{i\to-\infty} \frac{u_{i-1}^-}{u_i^+};\]
    \[\lim_{i\to\infty} \frac{u_{i-1}^-}{u_i^+}=\frac{bc-\sqrt{bc(bc-4)}}{2b}=\lim_{i\to-\infty} \frac{u_i^-}{u_{i-1}^+}.\]
    Moreover, the limits in the first line converge monotonically from above and the limits in the second line converge monotonically from below.
  \end{lemma}
  \begin{remark}
    The analogous limits with $\varepsilon=+$ and $\varepsilon=-$ reversed are obtained from these by interchanging the roles of $b$ and $c$.
  \end{remark}
  \begin{proof}
    When $bc=4$ it is easy to compute closed formulas for $u_i^\varepsilon$ and the claimed limits follow; we thus concentrate on the case $bc>4$.

    The standard Chebyshev polynomials (normalized, of the second kind) are defined by the recursion $u_0=0$, $u_1=1$, $u_{i+1}=ru_i-u_{i-1}$, which can be computed explicitly as
    \[u_i(r)=\frac{1}{2^i\sqrt{r^2-4}}\left(\big(r+\sqrt{r^2-4}\big)^i-\big(r-\sqrt{r^2-4}\big)^i\right).\]
    An induction on $i$ shows that 
    \[u_i^\varepsilon=\begin{cases} \frac{\sqrt{b}}{\sqrt{c}}u_i(\sqrt{bc}) & \text{if $i$ is even and $\varepsilon=+$;}\\ \frac{\sqrt{c}}{\sqrt{b}}u_i(\sqrt{bc}) & \text{if $i$ is even and $\varepsilon=-$;}\\ u_i(\sqrt{bc}) & \text{if $i$ is odd.} \end{cases}\]
    It follows that $u_i^\varepsilon$ can be computed explicitly as
    \[u_i^\varepsilon=\begin{cases} \frac{\sqrt{b}}{2^i\sqrt{c(bc-4)}}\left(\big(\sqrt{bc}+\sqrt{bc-4}\big)^i-\big(\sqrt{bc}-\sqrt{bc-4}\big)^i\right) & \text{if $i$ is even and $\varepsilon=+$;}\\ \frac{\sqrt{c}}{2^i\sqrt{b(bc-4)}}\left(\big(\sqrt{bc}+\sqrt{bc-4}\big)^i-\big(\sqrt{bc}-\sqrt{bc-4}\big)^i\right) & \text{if $i$ is even and $\varepsilon=-$;}\\ \frac{1}{2^i\sqrt{bc-4}}\left(\big(\sqrt{bc}+\sqrt{bc-4}\big)^i-\big(\sqrt{bc}-\sqrt{bc-4}\big)^i\right) & \text{if $i$ is odd.} \end{cases}\]

    For any $i\ne1$, we have
    \[\frac{\big(\sqrt{bc}+\sqrt{bc-4}\big)^i-\big(\sqrt{bc}-\sqrt{bc-4}\big)^i}{\big(\sqrt{bc}+\sqrt{bc-4}\big)^{i-1}-\big(\sqrt{bc}-\sqrt{bc-4}\big)^{i-1}}=\frac{\big(\sqrt{bc}+\sqrt{bc-4}\big)\cdot\left(1-\left(\frac{\sqrt{bc}-\sqrt{bc-4}}{\sqrt{bc}+\sqrt{bc-4}}\right)^i\right)}{1-\left(\frac{\sqrt{bc}-\sqrt{bc-4}}{\sqrt{bc}+\sqrt{bc-4}}\right)^{i-1}}.\]
    It follows that
    \[\lim_{i\to\infty} \frac{u_i^-}{u_{i-1}^+} = \lim_{i\to\infty} \left( \frac{\sqrt{c}}{2\sqrt{b}}\cdot\frac{\big(\sqrt{bc}+\sqrt{bc-4}\big)\cdot\left(1-\left(\frac{\sqrt{bc}-\sqrt{bc-4}}{\sqrt{bc}+\sqrt{bc-4}}\right)^i\right)}{1-\left(\frac{\sqrt{bc}-\sqrt{bc-4}}{\sqrt{bc}+\sqrt{bc-4}}\right)^{i-1}} \right) = \frac{\sqrt{c}}{2\sqrt{b}}\cdot\big(\sqrt{bc}+\sqrt{bc-4}\big),\]
    which is equivalent to the desired expression.
    Similarly, for $i\ne0$ we have
    \[\frac{\big(\sqrt{bc}+\sqrt{bc-4}\big)^{i-1}-\big(\sqrt{bc}-\sqrt{bc-4}\big)^{i-1}}{\big(\sqrt{bc}+\sqrt{bc-4}\big)^i-\big(\sqrt{bc}-\sqrt{bc-4}\big)^i}=\frac{\big(\sqrt{bc}-\sqrt{bc-4}\big)\cdot\left(1-\left(\frac{\sqrt{bc}-\sqrt{bc-4}}{\sqrt{bc}+\sqrt{bc-4}}\right)^{i-1}\right)}{4\left(1-\left(\frac{\sqrt{bc}-\sqrt{bc-4}}{\sqrt{bc}+\sqrt{bc-4}}\right)^i\right)},\]
    so that
    \[\lim_{i\to\infty} \frac{u_{i-1}^-}{u_i^+} = \lim_{i\to\infty} \left( \frac{2\sqrt{c}}{\sqrt{b}}\cdot\frac{\big(\sqrt{bc}-\sqrt{bc-4}\big)\cdot\left(1-\left(\frac{\sqrt{bc}-\sqrt{bc-4}}{\sqrt{bc}+\sqrt{bc-4}}\right)^{i-1}\right)}{4\left(1-\left(\frac{\sqrt{bc}-\sqrt{bc-4}}{\sqrt{bc}+\sqrt{bc-4}}\right)^i\right)} \right) = \frac{\sqrt{c}}{2\sqrt{b}}\cdot\big(\sqrt{bc}-\sqrt{bc-4}\big),\]
    which is again equivalent to the desired expression.
    This proves the claim for $i\to\infty$, the cases $i\to-\infty$ follow from these using $u_{-i}^\varepsilon=-u_i^\varepsilon$.

    For the final claim, observe that $\frac{u_{i+1}^{-\varepsilon}}{u_i^\varepsilon}<\frac{u_i^{-\varepsilon}}{u_{i-1}^\varepsilon}$ and $\frac{u_{i-1}^{-\varepsilon}}{u_i^\varepsilon}<\frac{u_i^{-\varepsilon}}{u_{i+1}^\varepsilon}$ are both equivalent to $u_{i+1}^{-\varepsilon}u_{i-1}^\varepsilon<u_i^{-\varepsilon}u_i^\varepsilon$.
    This inequality is then immediate, for $i>0$, from the following inductions:
    \begin{align*}
      u_{i+1}^{-}u_{i-1}^+
      =(cu_i^+-u_{i-1}^{-})u_{i-1}^+
      =cu_i^+ u_{i-1}^+-u_{i-1}^{-} u_{i-1}^+
      <cu_i^+ u_{i-1}^+-u_i^+ u_{i-2}^{-}
      =u_i^+(cu_{i-1}^+-u_{i-2}^{-})
      =u_i^+ u_i^{-};
      \\
      u_{i+1}^{+}u_{i-1}^-
      =(bu_i^--u_{i-1}^{+})u_{i-1}^-
      =bu_i^- u_{i-1}^--u_{i-1}^{+} u_{i-1}^-
      <bu_i^- u_{i-1}^--u_i^- u_{i-2}^{+}
      =u_i^-(bu_{i-1}^--u_{i-2}^{+})
      =u_i^- u_i^{+}.
    \end{align*}
    Again the case $i<0$ then follows from $u_{-i}^\varepsilon=-u_i^\varepsilon$.
  \end{proof}

\section{$\bfg$-vector Mutations}
\label{sec:tropical}

  We begin by studying transformations of $\RR^2$ which determine the change of $\bfg$-vectors when expanding an expression of the form~\eqref{eq:pointed} in terms of a cluster $\{x_k,x_{k+1}\}$.
  This adapts the notation from \cite[Definition 2.1.4]{Qin19} to our setting, see also \cite[Definition 4.1]{Rea14}.

  Write $\phi_0:\RR^2\to\RR^2$ for the identity map and define piecewise-linear maps $\phi_{\pm 1}:\RR^2\to\RR^2$ as follows:
  \begin{equation}
    \label{eq:forward mutation 1}
    \phi_1(\lambda)
    :=
    \begin{cases} 
      (-\lambda_0,c\lambda_0+\lambda_1) & \text{if $\lambda_0 \ge 0$;}\\
      (-\lambda_0,\lambda_1) & \text{if $\lambda_0 < 0$;}
    \end{cases}
    \qquad
    \phi_{-1}(\lambda)
    :=
    \begin{cases} 
      (\lambda_0,-\lambda_1) & \text{if $\lambda_1 \ge 0$;}\\
      (\lambda_0+b\lambda_1,-\lambda_1) & \text{if $\lambda_1 < 0$.}
    \end{cases}
  \end{equation}
  For $k\in\ZZ$ with $|k|>1$, define piecewise-linear maps
  \[\phi_k
    :=
    \begin{cases}
      (\phi_{-1}^{-1}\phi_1)^j & \text{if $k=2j$;}\\
      \phi_1(\phi_{-1}^{-1}\phi_1)^j & \text{if $k=2j+1$.}\\
    \end{cases}
  \]
  These determine the dominance region as we paraphrase from \cite[Section 3.1]{Qin19}.
  \begin{definition}
    \label{def:dominance}
    For $\lambda=(\lambda_0,\lambda_1)\in\ZZ^2$ and $k\in\ZZ$, define cones 
    \[\cC_k(\lambda)
      :=
      \begin{cases}
        \{(\lambda_0-r,\lambda_1+s):r,s\in\RR_{\ge0}\} & \text{if $k=2j$;}\\
        \{(\lambda_0+r,\lambda_1-s):r,s\in\RR_{\ge0}\} & \text{if $k=2j+1$.}
      \end{cases}
    \]
    The \emph{dominance region} $\cP_\lambda$ is the intersection  $\bigcap_{k\in\ZZ}\phi_k^{-1}\cC_k(\phi_k\lambda)$.
    When $\mu\in\cP_\lambda$, we say \emph{$\lambda$ dominates $\mu$}.
  \end{definition}

  We record here a few useful calculations relating to the tropical transformations $\phi_k$.
  First observe the following explicit expression for $\phi_2$:
  \begin{equation}
    \label{eq:forward two step mutation}
    \phi_2(\lambda)
    =
    \begin{cases}
      \big((bc-1)\lambda_0+b\lambda_1, -c\lambda_0-\lambda_1\big) & \text{if $\lambda_0\ge 0$ and $c\lambda_0+\lambda_1\ge 0$;}\\
      (-\lambda_0, -c\lambda_0-\lambda_1) & \text{if $\lambda_0\ge 0$ and $c\lambda_0+\lambda_1<0$;}\\
      (-\lambda_0+b\lambda_1, -\lambda_1) & \text{if $\lambda_0<0$ and $\lambda_1\ge 0$;}\\
      (-\lambda_0,-\lambda_1) & \text{if $\lambda_0<0$ and $\lambda_1<0$.}
    \end{cases}
  \end{equation}

  By an eigenvector of a piecewise-linear map $\phi$, we will mean a vector $\lambda$ such that there exists a positive scalar $\nu$ so that $\phi(\lambda)=\nu\lambda$.
  \begin{lemma}
    Any nonzero eigenvector of $\phi_2$ is a positive multiple of one of the vectors $\big(2b,-bc\pm\sqrt{bc(bc-4)}\big)\!$.
  \end{lemma}
  \begin{proof}
    First observe that, by equation \eqref{eq:forward two step mutation}, the equation $\phi_2(\lambda)=\nu\lambda$ cannot be satisfied with a positive $\nu$ unless $\lambda_0\ge 0$ and $c\lambda_0+\lambda_1\ge 0$.
    In this region, $\phi_2$ is linear with eigenvalue $\nu$ satisfying $\nu^2-(bc-2)\nu+1=0$, i.e.~$\nu=\frac{bc-2\pm\sqrt{bc(bc-4)}}{2}$.
    We thus require 
    \[\frac{bc-2\pm\sqrt{bc(bc-4)}}{2}\lambda_0=(bc-1)\lambda_0+b\lambda_1 \qquad\text{and}\qquad \frac{bc-2\pm\sqrt{bc(bc-4)}}{2}\lambda_1= -c\lambda_0-\lambda_1,\]
    or equivalently
    \[\frac{-bc\pm\sqrt{bc(bc-4)}}{2b}\lambda_0=\lambda_1 \qquad\text{and}\qquad \frac{-bc\mp\sqrt{bc(bc-4)}}{2c}\lambda_1=\lambda_0.\]
    As these represent the same relationship, the result immediately follows by inspection.
  \end{proof}

  Observe that the imaginary cone $\cI$ is spanned by the eigenvectors of $\phi_{2}$ and that $\phi_{2}$ is linear in $\cI$.
  \begin{lemma}
    \label{le:imaginary stability}
    For $j\in\ZZ$ and $\lambda\in\cI$, we have $\phi_{2j}(\lambda)\in\cI$.
  \end{lemma}
  \begin{proof}
    Since $\phi_{2j}=\phi_2^j$, the result follows from the case $j=1$ which is immediate.
  \end{proof}

  It will be useful to have explicit expressions for $\phi_k(\lambda)$ for $\lambda\in\cI$.
  \begin{lemma}
    \label{le:imaginary transformations}
    For $j\in\ZZ$ and $\lambda\in\cI$, we have
    \begin{align*}
      \phi_{2j}(\lambda)&=(u_{2j+1}^-\lambda_0+u_{2j}^+\lambda_1,-u_{2j}^-\lambda_0-u_{2j-1}^+\lambda_1);\\
      \phi_{2j+1}(\lambda)&=(-u_{2j+1}^-\lambda_0-u_{2j}^+\lambda_1,u_{2j+2}^-\lambda_0+u_{2j+1}^+\lambda_1).
    \end{align*}
  \end{lemma}
  \begin{proof}
    We work by induction on $j$, the case $j=0$ being clear from the definitions.
    For $\lambda\in\cI$, the action of~$\phi_2$ from \eqref{eq:forward two step mutation} can be rewritten as
    \[\phi_2(\lambda)=(u_3^-\lambda_0+u_2^+\lambda_1,-u_2^-\lambda_0-u_1^+\lambda_1).\]
    Therefore, after applying the equivalences for odd Chebyshev polynomials from Remark~\ref{rem:equivalences}, we have
    \begin{align*}
      \phi_{2j+2}(\lambda)&=\phi_2\phi_{2j}(\lambda)\\
      &=\big( (u_3^- u_{2j+1}^+-u_2^+u_{2j}^-)\lambda_0+(u_3^-u_{2j}^+-u_2^+u_{2j-1}^-)\lambda_1,-(u_2^-u_{2j+1}^+-u_1^+u_{2j}^-)\lambda_0-(u_2^-u_{2j}^+-u_1^+u_{2j-1}^-)\lambda_1 \big)\\
      &=(u_{2j+3}^+\lambda_0+u_{2j+2}^+\lambda_1, -u_{2j+2}^-\lambda_0-u_{2j+1}^-\lambda_1),
    \end{align*}
    where the last equality uses \eqref{eq:long Chebyshev recursion} with $i=2j+1$ and $\ell=2$.
    Using Remark~\ref{rem:equivalences} again, this is equivalent to the desired expression.

    Similarly, using $\phi_1(\lambda)=(-\lambda_0,u_2^-\lambda_0+u_1^+\lambda_1)$ for $\lambda\in\cI$ together with the basic Chebyshev recursion and the equality $\phi_{2j+1}=\phi_1\phi_2^j$ gives the claimed formula for $\phi_{2j+1}$ from that of $\phi_{2j}$.
  \end{proof}

\section{Proof of Theorem~\ref{th:dominance inequalities}}
\label{sec:dominance inequalities}

  Here we explicitly compute the dominance regions $\cP_\lambda$.
  The following Lemma proves Theorem~\ref{th:dominance inequalities} for~$\lambda\not\in\cI$.
  \begin{lemma}
    \label{le:cluster monomials}
    If $\lambda\in\ZZ^2\setminus\cI$, then $\cP_\lambda=\{\lambda\}$.
  \end{lemma}
  \begin{proof}
    Any such $\lambda$ is the $\bfg$-vector of a cluster monomial, say $x_k^{\alpha_k}x_{k+1}^{\alpha_{k+1}}$.
    In this case, the intersection
    \[\phi_{k-1}^{-1}\cC_{k-1}(\phi_{k-1}\lambda) \cap \phi_{k+1}^{-1}\cC_{k+1}(\phi_{k+1}\lambda)\]
    is precisely $\{\lambda\}$. 
    Indeed, for $k$ odd, the cone $\cC_{k-1}(\phi_{k-1}\lambda)$ lies entirely within a domain of linearity for $\phi_{k-1}^{-1}$.
    In particular, $\phi_{k-1}^{-1}\cC_{k-1}(\phi_{k-1}\lambda)$ is a cone containing $\lambda$ directed away from the origin (with walls parallel to the boundary of the domain of linearity containing $\lambda$).
    Next, again for $k$ odd, the cone $\cC_{k+1}(\phi_{k+1}\lambda)$ intersects the domain of linearity for $\phi_{k+1}^{-1}$ containing $\phi_{k+1}\lambda$ in a (possibly degenerate) convex quadrilateral with corners at the origin and $\phi_{k+1}\lambda$.
    In particular, the intersection of $\phi_{k+1}^{-1}\cC_{k+1}(\phi_{k+1}\lambda)$ with the cone containing $\lambda$ is a (possibly degenerate) convex quadrilateral with corners at the origin and $\lambda$.
    Combining these observations proves the result for $k$ odd, the case of even $k$ is similar.
  \end{proof}

  Next we aim to understand how inequalities transform under the action of a piecewise-linear map.
  The following well known fact about linear maps will suffice.

  \begin{lemma}
    \label{le:transformed inequalities}
    Let $M$ be an invertible $2\times 2$ matrix.
    Under the left action of $M$, say $M\mu=\mu'$ for $\mu,\mu'\in\RR^2$, the region inside $\RR^2$ defined by the inequality $\langle\alpha,\mu\rangle \le t$ with $\alpha\in\RR^2$ and $t\in\RR$ is transformed into the region defined by the inequality 
    \[
      \langle M^{-T}\alpha,\mu'\rangle\le t.
    \]
  \end{lemma}
  \begin{proof}
    This is immediate from the equalities
    \[
      \langle \alpha,\mu\rangle
      =
      \alpha^T\mu
      =
      \alpha^TM^{-1}M\mu
      =
      (M^{-T}\alpha)^TM\mu
      =
      \langle M^{-T}\alpha,\mu'\rangle.
    \]
  \end{proof}
 
  The following calculation is key to our main result.
  \begin{lemma}
    \label{le:dominance inequalities}
    For $k\in\ZZ$, $k\ne0$, and $\lambda\in\cI$, the region $\phi_k^{-1}\cC_k(\phi_k\lambda)\subset\RR^2$ is determined by the following inequalities when $k>0$:
    \begin{align}
      \label{ineq:1} u_k^-\mu_0+u_{k-1}^+\mu_1 &\le u_k^-\lambda_0+u_{k-1}^+\lambda_1;\\
      \label{ineq:2} u_{k+1}^-\mu_0+u_k^+\mu_1 &\le u_{k+1}^-\lambda_0+u_k^+\lambda_1;\\
      \label{ineq:3} -u_{k-1}^-\mu_0+u_k^+\mu_1 &\le u_{k+1}^-\lambda_0+u_k^+\lambda_1;\\
      \label{ineq:4} -u_{k-1}^-\mu_0-u_{k-2}^+\mu_1 &\le u_{k+1}^-\lambda_0+u_k^+\lambda_1;
    \end{align}
    and by the following inequalities when $k<0$:
    \begin{align*}
      u_{k+1}^-\mu_0+u_k^+\mu_1 &\le u_{k+1}^-\lambda_0+u_k^+\lambda_1;\\
      u_k^-\mu_0+u_{k-1}^+\mu_1 &\le u_k^-\lambda_0+u_{k-1}^+\lambda_1;\\
      u_k^-\mu_0-u_{k+1}^+\mu_1 &\le u_k^-\lambda_0+u_{k-1}^+\lambda_1;\\
      -u_{k+2}^-\mu_0-u_{k+1}^+\mu_1 &\le u_k^-\lambda_0+u_{k-1}^+\lambda_1.
    \end{align*}
    
\begin{figure}[h!]
  \centering
    \newcommand{\setconstants}{
      \tikzmath{\b = 2.5; \c =4.5 / \b; \s = sqrt(\b * \c * (\b * \c -4)); \bcmo = \b * \c - 1; \bcms = \b * \c - \s; \bcps = \b * \c + \s;}
      \tikzmath{ \xmin = -6; \xmax = 6; \ymin = -6; \ymax = 6;}
      \clip (\xmin,\ymin) rectangle (\xmax,\ymax);
    }
    \newcommand{\boundingrays}{
      \draw [color=gray, thick] (2*\xmin,0) -- (2*\xmax,0);
      \draw [color=gray, thick] (0,2*\ymin) -- (0, 2*\ymax);
      \draw [color=gray, thick, dashed] (0,0) -- (2*\xmax, -\xmax * \bcms / \b); 
      \draw [color=gray, thick, dashed] (0,0) -- ( - 4 * \b * \ymin / \bcps, 2*\ymin); 
    }
  \begin{tikzpicture}[scale=.45]
    \begin{scope}[shift={(0,0)}]
      \setconstants;
      \tikzmath{ \t=4.2; \r=1; \q=1.6; \xx={2*(\r+\q)*\b }; \yy = {-\r*\bcps -\q*\bcms}; \x={\t*\xx/sqrt((\xx)^2+(\yy)^2)}; \y={\t*\yy/sqrt((\xx)^2+(\yy)^2)};  }
      \fill[fill=black!50] (\xmax, {\c*(\x-\xmax)/(\b*\c-1)+\y}) --  (\x,\y) -- ({\x+\b*(\b*\c-2)*\y/(\b*\c-1)},0) -- (0,{-(\b*\c-1)*\x/\b-(\b*\c-2)*\y}) -- (\xmax, {-(\b*\c-1)*\x/\b-(\b*\c-2)*\y-\xmax/\b});
	    \fill[fill=black!25, opacity=0.3, domain=\xmin:\xmax, smooth, variable=\p] plot ({\p}, {\c*(\x-\p)/(\b*\c-1)+\y}) -- (\xmax,\ymax) -- (\xmin,\ymax) -- cycle;
      \fill[fill=black!25, opacity=0.3, domain=\xmin:\xmax, smooth, variable=\p] plot ({\p}, {(\b*\c-1)/(\b*\c-2)/\b*(\x-\p)+\y}) -- (\xmax,\ymax) -- (\xmin,\ymax) -- cycle;
      \fill[fill=black!25, opacity=0.3, domain=\xmin:\xmax, smooth, variable=\p] plot ({\p}, {(\b*\c-1)*(-\x+\p)/\b-(\b*\c-2)*\y}) -- (\xmax,\ymax) -- (\xmax,\ymin) -- cycle;
      \fill[fill=black!25, opacity=0.3, domain=\xmin:\xmax, smooth, variable=\p] plot ({\p}, {-((\b*\c-1)*\x+\b*(\b*\c-2)*\y+\p)/\b}) -- (\xmax,\ymin) -- (\xmin,\ymin) -- cycle;
      \boundingrays;
      \draw[domain=\xmin:\xmax, smooth, variable=\p, black!60, very thin] plot ({\p}, {\c*(\x-\p)/(\b*\c-1)+\y});
      \draw[domain=\xmin:\xmax, smooth, variable=\p, black!60, very thin] plot ({\p}, {(\b*\c-1)/(\b*\c-2)/\b*(\x-\p)+\y});
      \draw[domain=\xmin:\xmax, smooth, variable=\p, black!60, very thin] plot ({\p}, {(\b*\c-1)*(-\x+\p)/\b-(\b*\c-2)*\y});
      \draw[domain=\xmin:\xmax, smooth, variable=\p, black!60, very thin] plot ({\p}, {-((\b*\c-1)*\x+\b*(\b*\c-2)*\y+\p)/\b});
      \fill (\x , \y)  circle (4pt) node[below]{$\lambda$};
      \draw[thick] (\xmax, {\c*(\x-\xmax)/(\b*\c-1)+\y}) --  (\x,\y) -- ({\x+\b*(\b*\c-2)*\y/(\b*\c-1)},0) -- (0,{-(\b*\c-1)*\x/\b-(\b*\c-2)*\y}) -- (\xmax, {-(\b*\c-1)*\x/\b-(\b*\c-2)*\y-\xmax/\b});
    \end{scope}
    \draw (0,-7) node{$\phi_{-3}^{-1}\cC_{-3}(\phi_{-3}\lambda)$};
    \begin{scope}[shift={(20,0)}]
      \setconstants;
      \tikzmath{ \t=4.2; \r=1; \q=1.6; \xx={2*(\r+\q)*\b }; \yy = {-\r*\bcps -\q*\bcms}; \x={\t*\xx/sqrt((\xx)^2+(\yy)^2)}; \y={\t*\yy/sqrt((\xx)^2+(\yy)^2)};  }
      \fill[fill=black!50] ({\x+\b*(\y-\ymin)/(\b*\c-1)}, \ymin) -- (\x,\y) -- (0, {\c*(\b*\c-2)*\x/(\b*\c-1)+\y}) -- ({-(\b*\c-2)*\x-(\b*\c-1)*\y/\c},0) -- ({-\ymin/\c-(\b*\c-2)*\x-(\b*\c-1)*\y/\c},\ymin);
      \fill[fill=black!25, opacity=0.3, domain=\xmin:\xmax, smooth, variable=\p]plot ({\p}, {(\b*\c-1)*(\x-\p)/\b+\y}) -- (\xmin,\ymin) -- (\xmin,\ymax) -- cycle;
      \fill[fill=black!25, opacity=0.3, domain=\xmin:\xmax, smooth, variable=\p]plot ({\p}, {\c*(\b*\c-2)/(\b*\c-1)*(\x-\p)+\y}) -- (\xmin,\ymin) -- (\xmin,\ymax) -- cycle;
      \fill[fill=black!25, opacity=0.3, domain=\xmin:\xmax, smooth, variable=\p]plot ({\p}, {\c*(\b*\c-2)/(\b*\c-1)*\x+\c/(\b*\c-1)*\p +\y}) -- (\xmax,\ymin) -- (\xmin,\ymin) -- cycle;
      \fill[fill=black!25, opacity=0.3, domain=\xmin:\xmax, smooth, variable=\p]plot ({\p}, {-\c*(\b*\c-2)*\x-\c*\p -(\b*\c-1)*\y}) -- (\xmax,\ymin) -- (\xmax,\ymax) -- (\xmin,\ymax) -- cycle;
      \boundingrays;
      \draw[domain=\xmin:\xmax, smooth, variable=\p, black!60, very thin] plot ({\p}, {(\b*\c-1)*(\x-\p)/\b+\y});
      \draw[domain=\xmin:\xmax, smooth, variable=\p, black!60, very thin] plot ({\p}, {\c*(\b*\c-2)/(\b*\c-1)*(\x-\p)+\y});
      \draw[domain=\xmin:\xmax, smooth, variable=\p, black!60, very thin] plot ({\p}, {\c*(\b*\c-2)/(\b*\c-1)*\x+\c/(\b*\c-1)*\p +\y});
      \draw[domain=\xmin:\xmax, smooth, variable=\p, black!60, very thin] plot ({\p}, {-\c*(\b*\c-2)*\x-\c*\p -(\b*\c-1)*\y});
      \fill (\x , \y)  circle (4pt) node[right]{$\lambda$};
      \draw[thick] ({\x+\b*(\y-\ymin)/(\b*\c-1)}, \ymin) -- (\x,\y) -- (0, {\c*(\b*\c-2)*\x/(\b*\c-1)+\y}) -- ({-(\b*\c-2)*\x-(\b*\c-1)*\y/\c},0) -- ({-\ymin/\c-(\b*\c-2)*\x-(\b*\c-1)*\y/\c},\ymin);
    \end{scope}
    \draw (20,-7) node{$\phi_{3}^{-1}\cC_{3}(\phi_{3}\lambda)$};
  \end{tikzpicture}
\end{figure}

  \end{lemma}
  \begin{proof}
    We prove the claim for $k>0$, the proof for $k<0$ is similar or can be deduced from the other case by a symmetry argument.
    Following Lemma~\ref{le:imaginary transformations}, we consider even and odd sequences of mutations separately.

    For $k=2j$, $j>0$, and $\lambda\in\cI$, we observe that $\cC_k(\phi_k\lambda)\subset\RR^2$ is given by the inequalities 
    \[ (\dagger)\ \mu_0 \le u_{2j+1}^-\lambda_0+u_{2j}^+\lambda_1 \qquad\text{and}\qquad (\ddagger)\ -\mu_1\le u_{2j}^-\lambda_0+u_{2j-1}^+\lambda_1. \]
    We compute the region $\phi_{2j}^{-1}\cC_{2j}(\phi_{2j}\lambda)$ using Lemma~\ref{le:transformed inequalities} and the equality $\phi_{2j}^{-1}=\left(\phi_2^{-1}\right)^{j}$.
    First observe that~$\phi_2^{-1}=\phi_{-2}$ is given as follows:
    \begin{equation}
      \label{eq:backward two step mutation}
      \phi_{-2}(\lambda)
      =
      \begin{cases}
        \big(-\lambda_0-b\lambda_1, c\lambda_0+(bc-1)\lambda_1\big) & \text{if $\lambda_1\le 0$ and $\lambda_0+b\lambda_1\le 0$;}\\
        (-\lambda_0-b\lambda_1, -\lambda_1) & \text{if $\lambda_1\le 0$ and $\lambda_0+b\lambda_1>0$;}\\
        (-\lambda_0, c\lambda_0-\lambda_1) & \text{if $\lambda_1>0$ and $\lambda_0\le 0$;}\\
        (-\lambda_0,-\lambda_1) & \text{if $\lambda_1>0$ and $\lambda_0>0$.}
      \end{cases}
    \end{equation}
  
    \subsubsection*{Claim:} For $i>0$, $\phi_{2i}^{-1}\cC_{2j}(\phi_{2j}\lambda)$ is the region determined by the inequalities 
    \begin{align*}
      \tag{a} u_{2i}^-\mu_0+u_{2i-1}^+\mu_1 &\le u_{2j}^-\lambda_0+u_{2j-1}^+\lambda_1;\\
      \tag{b} u_{2i+1}^-\mu_0+u_{2i}^+\mu_1 &\le u_{2j+1}^-\lambda_0+u_{2j}^+\lambda_1;\\
      \tag{c} -u_{2i-1}^-\mu_0+u_{2i}^+\mu_1 &\le u_{2j+1}^-\lambda_0+u_{2j}^+\lambda_1;\\
      \tag{d} -u_{2i-1}^-\mu_0-u_{2i-2}^+\mu_1 &\le u_{2j+1}^-\lambda_0+u_{2j}^+\lambda_1.
    \end{align*}

    We prove the claim by induction on $i$.
    We see from \eqref{eq:backward two step mutation}, that $\lambda\in\cI$ implies the boundary ray for $\cC_{2j}(\phi_{2j}\lambda)$ corresponding to ($\ddagger$) lies entirely in the region in which $\phi_2^{-1}$ acts according to the matrix $\left[ \begin{array}{cc} -1 & -b\\ c & bc-1 \end{array}\right]$.
    By Lemma~\ref{le:transformed inequalities}, the inequality ($\ddagger$) transforms into the inequality $c\mu_0+\mu_1\le u_{2j}^-\lambda_0+u_{2j-1}^+\lambda_1$ which corresponds to (a) with $i=1$.
    Similarly, the boundary ray for $\cC_{2j}(\phi_{2j}\lambda)$ corresponding to ($\dagger$) intersects the three domains of linearity in which $\phi_2^{-1}$ acts according to the matrices $\left[ \begin{array}{cc} -1 & -b\\ c & bc-1 \end{array}\right]$, $\left[ \begin{array}{cc} -1 & -b\\ 0 & -1 \end{array}\right]$, $\left[ \begin{array}{cc} -1 & 0\\ 0 & -1 \end{array}\right]$.
    By Lemma~\ref{le:transformed inequalities}, the inequality~($\dagger$) can be seen to transform by these into each of inequalities (b), (c), (d) with~$i=1$.
    This establishes the base of our induction.

    Assuming the inequalities (a)-(d) hold for $i$, we apply Lemma~\ref{le:transformed inequalities} for $\phi_2^{-1}$.
    Both of the boundary rays corresponding to the inequalities (a) and (d) lie entirely in the region where $\phi_2^{-1}$ acts according to the matrix~$\left[ \begin{array}{cc} -1 & -b\\ c & bc-1 \end{array}\right]$, also the boundary segment corresponding to (b) intersects this region.
    Thus applying Lemma~\ref{le:transformed inequalities} to (a) gives the inequality 
    \[u_{2i+2}^-\mu_0+u_{2i+1}^+\mu_1=(u_3^+u_{2i}^--u_2^-u_{2i-1}^+)\mu_0+(u_2^+u_{2i}^--u_1^-u_{2i-1}^+)\mu_1\le u_{2j}^-\lambda_0+u_{2j-1}^+\lambda_1,\]
    which is the inequality (a) for $i+1$ by Lemma~\ref{eq:long Chebyshev recursion}; while applying this to (d) gives the inequality 
    \[-u_{2i+1}^-\mu_0-u_{2i}^+\mu_1=(-u_3^+u_{2i-1}^-+u_2^-u_{2i-2}^+)\mu_0+(-u_2^+u_{2i-1}^-+u_1^-u_{2i-2}^+)\mu_1\le u_{2j+1}^-\lambda_0+u_{2j}^+\lambda_1,\]
    which is the inequality (d) for $i+1$ again by Lemma~\ref{eq:long Chebyshev recursion}; finally applying this to (b) gives the inequality 
    \[u_{2i+3}^-\mu_0+u_{2i+2}^+\mu_1=(u_3^+u_{2i+1}^--u_2^-u_{2i}^+)\mu_0+(u_2^+u_{2i+1}^--u_1^-u_{2i}^+)\mu_1\le u_{2j+1}^-\lambda_0+u_{2j}^+\lambda_1,\]
    which is the inequality (b) for $i+1$.
    Similarly, the boundary segment corresponding to (c) lies entirely in the region where $\phi_2^{-1}$ acts according to the matrix $\left[ \begin{array}{cc} -1 & 0\\ c & -1 \end{array}\right]$.
    Thus applying Lemma~\ref{le:transformed inequalities} to (c) gives the inequality 
    \[-u_{2i+1}^-\mu_0-u_{2i}^+\mu_1=(u_1^+u_{2i-1}^--u_2^-u_{2i}^+)\mu_0+(-u_0^+u_{2i-1}^--u_1^-u_{2i}^+)\mu_1\le u_{2j+1}^-\lambda_0+u_{2j}^+\lambda_1,\]
    which is the inequality (d) for $i+1$ by Lemma~\ref{eq:long Chebyshev recursion}, in particular we see that the segment determined by (c) and the ray determined by (d) align in the image.
    Lastly, the boundary segment corresponding to (b) also intersects the regions where $\phi_2^{-1}$ acts according to the matrices $\left[ \begin{array}{cc} -1 & -b\\ 0 & -1 \end{array}\right]$ and $\left[ \begin{array}{cc} -1 & 0\\ 0 & -1 \end{array}\right]$ respectively.
    Applying Lemma~\ref{le:transformed inequalities} to (b) with the first matrix gives the inequality 
    \[-u_{2i+1}^-\mu_0+u_{2i+2}^+\mu_1=(-u_1^+u_{2i+1}^-+u_0^-u_{2i}^+)\mu_0+(u_2^+u_{2i+1}^--u_1^-u_{2i,+})\mu_1\le u_{2j+1}^-\lambda_0+u_{2j}^+\lambda_1,\]
    which is the inequality (c) for $i+1$ by Lemma~\ref{eq:long Chebyshev recursion}, while applying Lemma~\ref{le:transformed inequalities} to (b) with the second matrix gives the inequality 
    \[-u_{2i+1}^-\mu_0-u_{2i}^+\mu_1\le u_{2j+1}^-\lambda_0+u_{2j}^+\lambda_1,\]
    which again reproduces the inequality (d) and aligns with the previous segment and ray in the image.
    This completes the induction on $i$, proving the Claim and the result for $k$ even.

    For $k=2j+1$, $j\ge0$, and $\lambda\in\cI$, we get $\cC_{2j+1}(\phi_{2j+1}\lambda)\subset\RR^2$ is given by the inequalities 
    \[ (\dagger')\ -\mu_0\le u_{2j+1}^-\lambda_0+u_{2j}^+\lambda_1 \qquad\text{and}\qquad (\ddagger')\ \mu_1\le u_{2j+2}^-\lambda_0+u_{2j+1}^+\lambda_1.\]
    Using that $\phi_{2j+1}^{-1}=\left(\phi_2^{-1}\right)^j\phi_1^{-1}$, we compute the image inductively as above.
    From \eqref{eq:forward mutation 1} and Lemma~\ref{le:imaginary stability}, we see that the boundary ray for $\cC_{2j+1}(\phi_{2j+1}\lambda)$ corresponding to ($\dagger'$) lies entirely in the region in which $\phi_1^{-1}$ acts according to the matrix $\left[ \begin{array}{cc} -1 & 0\\ c & 1 \end{array}\right]$.
    Thus applying Lemma~\ref{le:transformed inequalities}, the inequality ($\dagger'$) is transformed by $\phi_1^{-1}$ into the inequality $\mu_0\le u_{2j+1}^-\lambda_0+u_{2j}^+\lambda_1$.
    The boundary ray corresponding to ($\ddagger'$) intersects both domains of linearity for $\phi_1^{-1}$ and thus produces the inequalities
    \[ c\mu_0+\mu_1\le u_{2j+2}^-\lambda_0+u_{2j+1}^+\lambda_1 \qquad\text{and}\qquad \mu_1\le u_{2j+2}^-\lambda_0+u_{2j+1}^+\lambda_1.\]

    \subsubsection*{Claim:} For $i\ge 0$, $\phi_{2i+1}^{-1}\cC_{2j+1}(\phi_{2j+1}\lambda)$ is the region determined by the inequalities 
    \begin{align*}
      \tag{a$'$} u_{2i+1}^-\mu_0+u_{2i}^+\mu_1 &\le u_{2j+1}^-\lambda_0+u_{2j}^+\lambda_1;\\
      \tag{b$'$} u_{2i+2}^-\mu_0+u_{2i+1}^+\mu_1 &\le u_{2j+2}^-\lambda_0+u_{2j+1}^+\lambda_1;\\
      \tag{c$'$} -u_{2i}^-\mu_0+u_{2i+1}^+\mu_1 &\le u_{2j+2}^-\lambda_0+u_{2j+1}^+\lambda_1;\\
      \tag{d$'$} -u_{2i}^-\mu_0-u_{2i-1}^+\mu_1 &\le u_{2j+2}^-\lambda_0+u_{2j+1}^+\lambda_1.
    \end{align*}
    By essentially the same calculations as above, these inequalities reproduce under the action of $\phi_2^{-1}$ and this completes the proof.
  \end{proof}

  We are now ready to prove Theorem~\ref{th:dominance inequalities}.
  The dominance region $\cP_\lambda=\bigcap_{k\in\ZZ}\phi_k^{-1}\cC_k(\phi_k\lambda)$ for $\lambda\in\cI$ is obtained by imposing all of the inequalities from Lemma~\ref{le:dominance inequalities} together with $\mu_0 -\lambda_0 \le 0$ and $0 \le \mu_1 - \lambda_1$ coming from $k=0$.

  We rewrite the inequalities from Lemma~\ref{le:dominance inequalities} using the Chebyshev recursion.
  For $k>0$, we get
  \begin{align*}
    u_k^-(\mu_0-\lambda_0)+u_{k-1}^+(\mu_1-\lambda_1) &\le 0;\\
    u_{k+1}^-(\mu_0-\lambda_0)+u_k^+(\mu_1-\lambda_1) &\le 0;\\
    -u_{k-1}^-(\mu_0-\lambda_0)+u_k^+(\mu_1-\lambda_1) &\le cu_k^+\lambda_0;\\
    -u_{k-1}^-(\mu_0-\lambda_0)-u_{k-2}^+(\mu_1-\lambda_1) &\le cu_k^+\lambda_0+bu_{k-1}^-\lambda_1;
  \end{align*}
  and, for $k<0$, we get
  \begin{align*}
    u_{k+1}^-(\mu_0-\lambda_0)+u_k^+(\mu_1-\lambda_1) &\le 0;\\
    u_k^-(\mu_0-\lambda_0)+u_{k-1}^+(\mu_1-\lambda_1) &\le 0;\\
    u_k^-(\mu_0-\lambda_0)-u_{k+1}^+(\mu_1-\lambda_1) &\le bu_k^-\lambda_1;\\
    -u_{k+2}^-(\mu_0-\lambda_0)-u_{k+1}^+(\mu_1-\lambda_1) &\le cu_{k+1}^+\lambda_0+bu_k^-\lambda_1.
  \end{align*}
  The first inequality in each list is redundant so we drop them.
  Moreover, for $k=1$ the third and fourth inequalities in the first list are the same so we can increment $k$ in the last equality without losing any information and similarly for $k=-1$ in the second list. 
  This gives
  \begin{align*}
    0 \le -u_{k+1}^-(\mu_0-\lambda_0)+u_k^-(\mu_1-\lambda_1); &\\
    -u_{k-1}^-(\mu_0-\lambda_0)+u_k^+(\mu_1-\lambda_1) &\le cu_k^+\lambda_0;\\
    -u_k^-(\mu_0-\lambda_0)-u_{k-1}^+(\mu_1-\lambda_1) &\le cu_{k+1}^+\lambda_0+bu_k^-\lambda_1;
  \end{align*}
  for $k>0$ and
  \begin{align*}
    0 \le -u_k^-(\mu_0-\lambda_0)+u_{k-1}^-(\mu_1-\lambda_1); &\\
    u_k^-(\mu_0-\lambda_0)-u_{k+1}^+(\mu_1-\lambda_1) &\le bu_k^-\lambda_1;\\
    -u_{k+1}^-(\mu_0-\lambda_0)-u_k^+(\mu_1-\lambda_1) &\le cu_k^+\lambda_0+bu_{k-1}^-\lambda_1;
  \end{align*}
  for $k<0$.

  Then, using that $-u_k^\varepsilon<0$ for $k>0$ and $u_k^\varepsilon<0$ for $k<0$, we rewrite the inequalities again as
  \begin{align}
    \label{ieq1} 0 \le -(\mu_0-\lambda_0)-\frac{u_k^+}{u_{k+1}^-}(\mu_1-\lambda_1); & \\
    \label{ieq2} -\frac{u_{k-1}^-}{u_k^+}(\mu_0-\lambda_0)+(\mu_1-\lambda_1) & \le c\lambda_0;\\
    \label{ieq3} -(\mu_0-\lambda_0)-\frac{u_{k-1}^+}{u_k^-}(\mu_1-\lambda_1) &\le c\frac{u_{k+1}^+}{u_k^-}\lambda_0+b\lambda_1;
  \end{align}
  for $k>0$ and
  \begin{align}
    \label{ieq4} 0 \le (\mu_0-\lambda_0)+\frac{u_{k-1}^+}{u_k^-}(\mu_1-\lambda_1); & \\
    \label{ieq5} -(\mu_0-\lambda_0)+\frac{u_{k+1}^+}{u_k^-}(\mu_1-\lambda_1) &\le -b\lambda_1;\\
    \label{ieq6} -\frac{-u_{k+1}^-}{u_k^+}(\mu_0-\lambda_0)+(\mu_1-\lambda_1) &\le -c\lambda_0-b\frac{u_{k-1}^-}{u_k^+}\lambda_1;
  \end{align}
  for $k<0$.
  We now study each sequence of inequalities in turn.

  For $\mu_1-\lambda_1\ge 0$, the inequalities \eqref{ieq1} become more restrictive as $k$ gets larger since the sequence $\frac{-u_{k+1}^-}{u_k^+}$ of negative slopes is monotonically increasing (cf. Lemma~\ref{le:limits}).
  Thus, following Lemma~\ref{le:limits}, in the limit as $k\to\infty$, we obtain the inequality below determining a boundary of $\cP_\lambda$:
  \[ 0 \le -(\mu_0-\lambda_0)-\frac{bc-\sqrt{bc(bc-4)}}{2c}(\mu_1-\lambda_1). \]

  Similarly, the inequalities \eqref{ieq2} become more restrictive for $\mu_0\le0$ and $\mu_1\ge0$ as $k$ gets larger since the sequence $\frac{u_{k-1}^-}{u_k^+}$ of positive slopes is monotonically increasing and the intersection with \eqref{ieq1} moves lower on the $\mu_1$-axis as $k$ increases.
  Thus, following Lemma~\ref{le:limits}, in the limit as $k\to\infty$, we obtain the inequality
  \[ -\frac{bc-\sqrt{bc(bc-4)}}{2b}(\mu_0-\lambda_0)+(\mu_1-\lambda_1) \le c\lambda_0 \]
  determining a boundary of $\cP_\lambda$.
  Observe further that the inequalities $\mu_0-\lambda_0 \le 0$ and $0 \le \mu_1-\lambda_1$ allow to strengthen this as
  \[ 0 \le -\frac{bc-\sqrt{bc(bc-4)}}{2b}(\mu_0-\lambda_0)+(\mu_1-\lambda_1) \le c\lambda_0. \]

  Finally, the inequalities \eqref{ieq3} become more restrictive for $\mu_1\le0$ as $k$ gets larger since the sequence $\frac{-u_k^-}{u_{k-1}^+}$ of negative slopes is monotonically increasing (cf. Lemma~\ref{le:limits}) and the intersection with \eqref{ieq2} (for $k+1$) moves to the right on the $\mu_0$-axis as $k$ increases.
  Thus, following Lemma~\ref{le:limits}, in the limit as $k\to\infty$, we obtain the inequality below determining a boundary of $\cP_\lambda$:
  \[ -(\mu_0-\lambda_0)-\frac{bc-\sqrt{bc(bc-4)}}{2c}(\mu_1-\lambda_1) \le \frac{bc+\sqrt{bc(bc-4)}}{2c}c\lambda_0+b\lambda_1. \]

  Similar arguments using \eqref{ieq4}-\eqref{ieq6} lead to the remaining inequalities determining the boundary of $\cP_\lambda$:
  \[ 0 \le (\mu_0-\lambda_0)+\frac{bc+\sqrt{bc(bc-4)}}{2c}(\mu_1-\lambda_1);\]
  \[ 0 \le -(\mu_0-\lambda_0)+\frac{bc-\sqrt{bc(bc-4)}}{2c}(\mu_1-\lambda_1) \le -b\lambda_1;\]
  \[ 0 \le \frac{bc-\sqrt{bc(bc-4)}}{2b}(\mu_0-\lambda_0)+(\mu_1-\lambda_1) \le -c\lambda_0-\frac{bc+\sqrt{bc(bc-4)}}{2b}b\lambda_1.\]
  These can easily be seen to be equivalent to the remaining inequalities from Theorem~\ref{th:dominance inequalities} and this complete the proof.

\section{Proof of Corollary~\ref{cor:dominance vertices}}
\label{sec:dominance vertices}

  This follows from basic manipulations finding the intersection points of the boundary segments determined by the inequalities from Theorem~\ref{th:dominance inequalities}.
  Note that in each of the cases (2), (4), and (6) there are only four inequalities to consider while in cases (3) and (5) there are only three inequalities to consider.
  We leave the details as an exercise for the reader.

\section{Proof of Theorem~\ref{th:maximum support}}
\label{sec:maximum support} 
  
  We begin observing that, by Lemma~\ref{le:cluster monomials}, the dominance region $\cP_\lambda$ of any $\bfg$-vector $\lambda\not\in\cI$ is just the point~$\lambda$.
  Therefore, by Theorem~\ref{th:dominance} the corresponding pointed element is the cluster monomial whose $\bfg$-vector is $\lambda$. 
  The support of this element is exactly $\cS_\lambda$ by \cite[Proposition 4.1]{LLZ14}.
  
  Every pointed basis element for $\cA(b,c)$ admits an opposite $\bfg$-vector arising by interchanging the roles of $b$, $c$ and $x_0$, $x_1$ in \eqref{eq:pointed}. 
  One can easily compute the following correspondence.
  \begin{lemma}
    Given a $\bfg$-vector $\lambda$, the opposite $\bfg$-vector $\lambda'$ is obtained as follows:
    \begin{itemize}
      \item if $\lambda_0,\lambda_1\ge0$, then $\lambda'=\lambda$;
      \item if $\lambda_1\ge0$ and $\lambda_0<0$, then $\lambda'=(\lambda_0,-c\lambda_0+\lambda_1)$;
      \item if $\lambda_0>0$ and $\lambda_0+b\lambda_1>0$, then $\lambda'=(\lambda_0+b\lambda_1,\lambda_1)$;
      \item otherwise, $\lambda'=(\lambda_0+b\lambda_1,-c\lambda_0-(bc-1)\lambda_1)$.
    \end{itemize}
  \end{lemma}

  In particular, we obtain an opposite dominance region $\cP'_\lambda$ for each $\bfg$-vector $\lambda$.
  \begin{lemma}
    For $\lambda\in\cI$, the opposite dominance polygon $\cP'_\lambda$ pointed at its opposite $\bfg$-vector $\lambda'=(\lambda'_0,\lambda'_1)$ is the region consisting of those $\mu\in\RR^2$ satisfying $\mu_0 \geq \lambda'_0, \mu_1 \leq\lambda'_1$, and the following inequalities:
    {
      \everymath={\displaystyle}
      \def\arraystretch{2.8}
      \[
        \begin{array}{rcccl}
          0 & \leq & \frac{b c-\sqrt{b c (b c-4)}}{2 c}(\mu_1-\lambda'_1)+(\mu_0-\lambda'_0) & \leq & -b\lambda'_1-\frac{b c+\sqrt{b c (b c-4)}}{2c}c\lambda'_0
          \\
          0 & \leq & -\frac{b c-\sqrt{b c (b c-4)}}{2 c}(\mu_1-\lambda'_1)+(\mu_0-\lambda'_0) & \leq & b\lambda'_1
          \\
          0 & \leq &  -(\mu_1-\lambda'_1)-\frac{b c-\sqrt{b c (b c-4)}}{2 b}(\mu_0-\lambda'_0) & \leq & \frac{b c+\sqrt{b c (b c-4)}}{2b}b\lambda'_1+c\lambda'_0
          \\
          0 & \leq & (\mu_1-\lambda'_1) - \frac{b c-\sqrt{b c (b c-4)}}{2 b} (\mu_0-\lambda'_0) & \leq & -c \lambda'_0
        \end{array}
      \]
    }
  \end{lemma}
  \begin{proof}
    As stated above, we obtain the opposite $\bfg$-vector by interchanging $b$, $c$ and swapping the roles of $\lambda_0,\lambda_1$.
    Translating from Theorem~\ref{th:dominance inequalities}, it immediately follows that the opposite dominance region is given by the claimed inequalities.
  \end{proof}

  For $\lambda\in\cI$, the region $\cS_\lambda$ from Theorem~\ref{th:maximum support} is determined by the following inequalities:
  \begin{align*}
    \lambda_1 &\leq \mu_1;\\
    \lambda'_0 &\leq \mu_0;\\
    0 &\leq -\frac{b c+\sqrt{b c (b c-4)}}{2 b}(\mu_0-\lambda_0)-(\mu_1-\lambda_1);\\
    0 &\leq -\frac{b c-\sqrt{b c (b c-4)}}{2 b}(\mu_0-\lambda'_0) -(\mu_1-\lambda'_1).
  \end{align*}
  To begin proving that this is the maximal support we claim that $\cP_\lambda,\cP'_\lambda\subset\cS_\lambda$.

  Observe that the last inequality can be rewritten as
  \[ \frac{b c-\sqrt{b c (b c-4)}}{2 b}(\mu_0-\lambda_0) + (\mu_1-\lambda_1) \leq -c\lambda_0 - \frac{b c + \sqrt{b c (b c-4)}}{2 } \lambda_1,\]
  which together with the second to last inequality already gives two of the boundaries for the dominance region $\cP_\lambda$.
  Note then that, under the assumption $\mu_0\leq\lambda_0$, the inequality 
  \[0 \leq \frac{b c-\sqrt{b c (b c-4)}}{2 b}(\mu_0-\lambda_0)+(\mu_1-\lambda_1)\]
  defining another boundary of $\cP_\lambda$ is more restrictive than the inequality $\lambda_1\leq\mu_1$ bounding $\cS_\lambda$.

  It follows from Corollary~\ref{cor:dominance vertices} that the minimum value for $\mu_0$ inside $\cP_\lambda$ occurs when $\mu_1=0$, i.e. either at the point $\big( \lambda_0+\frac{bc+\sqrt{bc(bc-4)}}{2c}\lambda_1 , 0 \big)$ when $0 \leq c\lambda_0+2\lambda_1$ or at the point $-\frac{bc+\sqrt{bc(bc-4)}}{2c} \big( \frac{bc+\sqrt{bc(bc-4)}}{2b}\lambda_0+\lambda_1 , 0 \big)$ when $c\lambda_0+2\lambda_1<0$.
  The second point can be rewritten as
  \[\big(\lambda_0 -\frac{bc+\sqrt{bc(bc-4)}}{2c} c\lambda_0-\frac{bc+\sqrt{bc(bc-4)}}{2c} \lambda_1 , 0 \big)\]
  but, since $c\lambda_0 < -2\lambda_1$, the first coordinate is greater than $\lambda_0+\frac{bc+\sqrt{bc(bc-4)}}{2c}\lambda_1$.
  Then, using $\frac{bc+\sqrt{bc(bc-4)}}{2c}<b$ and $\lambda_1<0$, we see that the minimum value of $\mu_0$ inside $\cP_\lambda$ satisfies the inequality $\lambda'_0=\lambda_0+b\lambda_1 \leq \mu_0$.
  In particular, combining with the observation above, we see that $\cP_\lambda\subset\cS_\lambda$.

  Similarly, the second to last inequality can be rewritten as
  \[ \frac{b c-\sqrt{b c (b c-4)}}{2 c}(\mu_1-\lambda'_1)+(\mu_0-\lambda'_0) \leq -b\lambda'_1-\frac{b c+\sqrt{b c (b c-4)}}{2c}c\lambda'_0,\]
  which together with the last inequality already gives two of the boundaries for the opposite dominance region $\cP'_\lambda$.
  Note then that, under the assumption $\mu_1\leq\lambda'_1$, the inequality
  \[ 0 \leq \frac{b c-\sqrt{b c (b c-4)}}{2 c}(\mu_1-\lambda'_1)+(\mu_0-\lambda'_0) \]
  defining another boundary of $\cP_\lambda$ is more restrictive than the inequality $\lambda'_0\leq\mu_0$ bounding $\cS_\lambda$.
  As above, the minimum value of $\mu_1$ inside $\cP'_\lambda$ occurs when $\mu_0=0$, i.e. either at the point $\big(0, \frac{bc+\sqrt{bc(bc-4)}}{2b}\lambda'_0+\lambda'_1 \big)$
  when $0 \leq 2\lambda'_0+b\lambda'_1$ or at the point $-\frac{bc+\sqrt{bc(bc-4)}}{2b} \big(0, \lambda'_0+\frac{bc+\sqrt{bc(bc-4)}}{2c}\lambda'_1 \big)$
  when $2\lambda'_0+b\lambda'_1<0$.
  The second point can be rewritten as
  \[\big(0, -\frac{bc+\sqrt{bc(bc-4)}}{2b} \lambda'_0-\frac{bc+\sqrt{bc(bc-4)}}{2b} b\lambda'_1+\lambda'_1 \big)\]
  but, since $b\lambda'_1\le -2\lambda'_0$, the first coordinate is greater than $\frac{bc+\sqrt{bc(bc-4)}}{2b}\lambda'_0+\lambda'_1$.
  Then, using $\frac{bc+\sqrt{bc(bc-4}}{2b}<c$ and $\lambda'_0<0$, we see that the minimum value of $\mu_1$ inside $\cP'_\lambda$ satisfies the inequality $\lambda_1=c\lambda'_0+\lambda'_1\le\mu_1$ and so $\cP'_\lambda\subset\cS_\lambda$.

  To continue, we compare the support for an arbitrary basis element pointed at $\lambda\in\cI$ with the greedy basis elements having $\bfg$-vector inside the dominance polygon $\cP_\lambda$.

  The support region $\cG_\lambda$ of the greedy basis element with $\bfg$-vector $\lambda$ is well-known \cite{CGMMRSW17,LLZ14}.
  This is indicated by the solid and dotted lines in our figures where dotted lines indicate points that are excluded from the support.
  (The actual support of the greedy basis element consists of the points of the form $\lambda+(b \alpha_0 ,c \alpha_1)$, $\alpha_0,\alpha_1\in\ZZ$, inside $\cG_\lambda$.)
  For our purposes, it is enough to observe the following closure property for the greedy support region which is evident from the figure in Theorem~\ref{th:maximum support}.
  \begin{definition}
    A subset $\cS\subset\RR^2$ is called \scalebox{1.7}{$\llcorner$}-closed if $\lambda,\mu\in\cS$ implies the segments joining $\lambda$ and $\mu$ to~$\big(\!\min(\lambda_0,\mu_0),\min(\lambda_1,\mu_1)\big)$ are contained in $\cS$.
  \end{definition}
  In particular, (an upper bound for) the greedy support region $\cG_\lambda$ can be found using its $\bfg$-vector $\lambda$ 
  and its opposite $\bfg$-vector $\lambda'$. 
  Write $\overline{\cG}_\lambda$ for the downward scaling of the \scalebox{1.7}{$\llcorner$}-closure of $\{\lambda,\lambda'\}$, that is $\overline{\cG}_\lambda$ contains the segments joining $\lambda$ and $\lambda'$ with $\big(\!\min(\lambda_0,\lambda'_0),\min(\lambda_1,\lambda'_1)\big)$ and for any $\mu\in\overline{\cG}_\lambda$ we have $t\mu\in\overline{\cG}_\lambda$ for $0\le t\le 1$.
  Clearly, $\cG_\lambda\subset\overline{\cG}_\lambda$.

  As we saw above, for any $\bfg$-vector $\mu\in\cP_\lambda$ its opposite $\bfg$-vector $\mu'$ is also contained in $\cS_\lambda$.
  But $\cS_\lambda$ is \scalebox{1.7}{$\llcorner$}-closed and closed under downward scaling. 
  It follows that $\overline{\cG}_\mu\subset\cS_\lambda$ for any $\mu\in\cP_\lambda$ and hence, following Theorem~\ref{th:dominance}, the support of any basis element pointed at $\lambda$ is contained in $\cS_\lambda$.

  In particular, $\cG_\lambda\subset\cS_\lambda$.
  Note also that the dominance region $\cP_\lambda$ contains the intersection of $\cS_\lambda$ with the region $\cR$ defined by the inequalities $\mu_0\ge0$ and $\lambda_1\mu_0-\lambda_0\mu_1\ge0$.
  Similarly, the opposite dominance region $\cP'_\lambda$ contains the intersection of $\cS_\lambda$ with the region $\cR'$ defined by the inequalities $\mu_1\ge0$ and $-\lambda'_1\mu_0+\lambda'_0\mu_1\ge0$.
  But $\cS_\lambda=\cG_\lambda\cup(\cS_\lambda\cap\cR)\cup(\cS_\lambda\cap\cR')$ so $\cS_\lambda$ is the maximum possible support for an element pointed at~$\lambda$.

\section{The untwisted affine case}
\label{sec:affine}

  In this section we compare our main result with the construction of \cite{RSW19} in the case $b=c=2$.
  To match conventions, in this section, we work over an algebraically closed field $\kk$ of characteristic 0.
  Because the exchange matrix is full rank there is no loss of generality in continuing to work in the coefficient-free case.
  We will identify the family of bases in Theorem \ref{th:dominance} with the continuous family of bases of $\cA(2,2)$ constructed in \cite{RSW19} from generalized minors, which we recast here with the current notation. 

  \begin{theorem}[{\cite[Theorem 4.6]{RSW19}}]
    \label{thm:rsw}
    Choose a point $\bfa^{(n)}=(a_1,\dots,a_n)\in(\kk^\times)^n$ for each $n\geq1$.
    Then, together with all cluster monomials, the elements 
    \begin{equation}
      \label{eq:generalized minor}
      x_{(n,-n)}^{\bfa^{(n)}}:= x_0^{-n} x_1^{-n} \sum_{0 \leq k \leq \ell \leq n} \sum_{r=0}^\ell {\ell-r\choose k} {n-2r\choose \ell-r} S_{\bfa^{(n)},r}  \, x_0^{2(\ell-k)} x_1^{2(n-\ell)}
    \end{equation}
    with
    \[
      S_{\bfa^{(n)},r}=\sum\limits_{\substack{I,J\subseteq[1,n]\\|I|=r=|J|\\ I\cap J=\varnothing}}\frac{\prod_{i\in I} a_i}{\prod_{j\in J} a_j}
    \]
    form a linear basis of $\cA(2,2)$.
  \end{theorem}

  We begin by noting that when $b=c=2$ the imaginary cone $\cI$ degenerates to the ray spanned by $(1,-1)$.
  For any $\lambda\in\cI$, the dominance region $\cP_\lambda$ is the segment connecting $\lambda$ to the origin.
  It follows that $\bfg$-vectors dominated by $\lambda=(n,-n)$ are of the form $(n-2r,-n+2r)$ for $0\leq r \leq n/2$.

  In order to use Theorem \ref{th:dominance} we need to fix a reference pointed basis of $\cA(2,2)$; to simplify our computations, we choose to work with the \emph{generic basis}.
  This basis consists of the cluster monomials of $\cA(2,2)$ together with the elements
  \[
    x_{(n,-n)}^{ge}
    :=
		\Big(x_0 x_1^{-1} + x_0^{-1}x_1^{-1} + x_0^{-1}x_1\Big)^n
		=
    x_0^{-n} x_1^{-n}
    \sum_{0 \leq k \leq \ell \leq n}
    {\ell\choose k} {n\choose \ell}
    \, x_0^{2(\ell-k)} x_1^{2(n-\ell)}.
  \]

  \begin{proposition}
    \label{prop:rewrite}
    For $n\ge0$, we have
    \[
      x_{(n,-n)}^{\bfa^{(n)}}
      =
      \sum_{r=0}^{\lfloor n/2 \rfloor}
      S_{\bfa^{(n)},r}
      \,
      x_{(n-2r,-n+2r)}^{ge}.
    \]
  \end{proposition}
  \begin{proof}
    To begin, we observe that \eqref{eq:generalized minor} may be rewritten as
    \[
      x_{(n,-n)}^{\bfa^{(n)}} = x_0^{-n} x_1^{-n} \sum_{k=0}^n \sum_{\ell=k}^n \sum_{r=0}^\ell {\ell-r\choose k} {n-2r\choose \ell-r} S_{\bfa^{(n)},r}  \, x_0^{2(\ell-k)} x_1^{2(n-\ell)}.
    \]
    The first binomial coefficient above is zero if $0 \le \ell-r < k$ while the second is zero for $\ell < r \le \lfloor n/2\rfloor$ or if $n-r < \ell$, therefore $x_{(n,-n)}^{\bfa^{(n)}}$ can be expressed as
    \begin{align*}
      x_{(n,-n)}^{\bfa^{(n)}}
      &=
      x_0^{-n} x_1^{-n} \sum_{k=0}^n \sum_{r=0}^{\lfloor n/2\rfloor} \sum_{\ell=k+r}^{n-r} {\ell-r\choose k} {n-2r\choose \ell-r} S_{\bfa^{(n)},r}  \, x_0^{2(\ell-k)} x_1^{2(n-\ell)}.
    \end{align*}
    But then, rearranging terms and replacing $\ell$ by $\ell+r$, this becomes
    \begin{align*}
      x_{(n,-n)}^{\bfa^{(n)}}
      &=
      \sum_{r=0}^{\lfloor n/2\rfloor} S_{\bfa^{(n)},r}\, x_0^{-n+2r} x_1^{-n+2r} \sum_{k=0}^n \sum_{\ell=k}^{n-2r} {\ell\choose k} {n-2r\choose \ell} \, x_0^{2(\ell-k)} x_1^{2(n-2r-\ell)}
      =
      \sum_{r=0}^{\lfloor n/2\rfloor} S_{\bfa^{(n)},r}\, x_{n-2r,-n+2r}^{ge}
    \end{align*}
    as desired.
  \end{proof}

  \begin{remark}
    \label{rk:symmetric}
    Observe that the expansion coefficients $S_{\bfa^{(n)},r}$ can be expressed as the ratio of monomial symmetric functions $\frac{m_{2^{(r)},1^{(n-2r)}}}{m_{1^{(n)}}}$ evaluated at $\bfa^{(n)}$.
    The analogous expansion coefficients when $x_{(n,-n)}^{\bfa^{(n)}}$ is expressed in terms of the triangular basis (resp. greedy basis) are the ratios of Schur functions $\frac{s_{2^{(r)},1^{(n-2r)}}}{s_{1^{(n)}}}$ (resp. ratios of elementary symmetric functions $\frac{e_{n-r,r}}{e_n}$) evaluated at $\bfa^{(n)}$.
    We leave the details to the reader.
  \end{remark}

  \begin{theorem}
    As the points $\bfa^{(n)}$ vary in $(\kk^\times)^n$, the bases in Theorem \ref{thm:rsw} recover precisely all the pointed bases of $\cA(2,2)$.
  \end{theorem}
  \begin{proof}
    By Theorem \ref{th:dominance}, in view of Proposition \ref{prop:rewrite} and the discussion immediately before it, it suffices to show that as $\bfa^{(n)}$ vary in $(\kk^\times)^n$ the tuple of coefficients $\big(S_{\bfa^{(n)},r}\big)_{1\leq r \leq \lfloor n/2\rfloor}$ assume all the values in $\kk^{\lfloor n/2\rfloor}$.
    (The fact that $S_{\bfa^{(n)},0}=1$ is immediate from the definition.)

    By Remark \ref{rk:symmetric}, this is equivalent to the fact that the monomial symmetric functions in $n$ variables $m_{2^{(r)},1^{(n-2r)}}$ are algebraically independent.
    But this follows immediately from the observation that 
    \[ 
      m_{2^{(r)},1^{(n-2r)}} = \sum_{i=0}^r \gamma_i\, e_{n-i,i}
    \]
    for some coefficients $\gamma_i$ with $\gamma_r=1$, and the Fundamental Theorem of Symmetric Polynomials. 
  \end{proof}

\bibliographystyle{amsalpha}
\bibliography{bibliography}

\end{document}